\documentclass[
12pt
]{article}

\usepackage{color}
\definecolor{LinkColor}{rgb}{0,0,1}
\definecolor{LinkColor2}{rgb}{1,0,0}
\definecolor{lbcolor}{rgb}{0.85,0.85,0.85}
\definecolor{FrameColor}{rgb}{0.85,0.85,0.85}

\usepackage[%
left=1in,
right=1in,
bottom=1in,
top=1in,
footskip=0.3in,
]{geometry}

\usepackage[ngerman, french, english]{babel}
\usepackage[utf8]{inputenc}
\usepackage[T1]{fontenc}
\usepackage{enumerate}
\usepackage{setspace}

\newcommand{\pskip}{\\[-3mm]}
\newcommand{\bpskip}{\\[-2mm]}

\usepackage[babel,french=guillemets,german=swiss]{csquotes}

\usepackage{graphicx}
\usepackage{caption}
\usepackage{amsmath}
\usepackage{amssymb}
\usepackage{dsfont}
\usepackage{nicefrac}
\usepackage{empheq}
\allowdisplaybreaks

\usepackage[%
pdftitle={Titel},
pdfauthor={Autor},
pdfcreator={LaTeX, LaTeX with hyperref and KOMA-Script},
pdfsubject={Betreff}, 
pdfkeywords={Keywords}
]{hyperref} 

\hypersetup{%
	colorlinks=true,
	linkcolor=LinkColor,%
	anchorcolor=LinkColor,%
	citecolor=LinkColor2,%
	filecolor=LinkColor,%
	menucolor=LinkColor,%
	pagecolor=LinkColor,%
	urlcolor=LinkColor,%
}

\usepackage[numbers]{natbib}
\usepackage{amsthm}

\newtheoremstyle{tstyle}
{15pt}	
{5pt}	
{\itshape}	
{}	
{\bfseries}	
{.}	
{0.5em}	
{}	

\theoremstyle{tstyle}

\newtheorem{theorem}{Theorem}
\newtheorem{lemma}[theorem]{Lemma}
\newtheorem{proposition}[theorem]{Proposition}
\newtheorem{corollary}[theorem]{Corollary}
\newtheorem{definition}[theorem]{Definition}
\newtheorem{remark}[theorem]{Remark}

\newtheoremstyle{cstyle}
{15pt}	
{5pt}	
{}	
{}	
{\bfseries}	
{}	
{0.2222em}	
{}	

\theoremstyle{cstyle}

\makeatletter
\g@addto@macro{\thm@space@setup}{\thm@headpunct{}}
\renewenvironment{proof}[1][\proofname]{\par
	\pushQED{\qed}%
	\normalfont \topsep6\p@\@plus6\p@\relax
	\trivlist
	\item[\hskip\labelsep
	\bfseries
	#1\@addpunct{\,}]\ignorespaces
}{%
	\popQED\endtrivlist\@endpefalse
}
\makeatother

\makeatletter
\g@addto@macro{\thm@space@setup}{\thm@headpunct{}}
\newenvironment{sketch-proof}[1][Sketch of the proof]{\par
	\pushQED{\qed}%
	\normalfont \topsep6\p@\@plus6\p@\relax
	\trivlist
	\item[\hskip\labelsep
	\bfseries
	#1\@addpunct{\,}]\ignorespaces
}{%
	\popQED\endtrivlist\@endpefalse
}
\makeatother

\newcommand{\RR}{\mathbb R}
\newcommand{\WW}{{\mathcal W}}
\newcommand{\HH}{{\mathcal H}}
\newcommand{\VV}{{\mathcal V}}
\newcommand{\NN}{\mathbb N}
\newcommand{\BB}{{\mathbb B_K}}

\newcommand{\IBB}{\mathring{\mathbb B}_{K}}

\newcommand{\IBBB}{\mathring{\mathbb B}_{2K}}

\newcommand{\MM}{\mathbb M}

\newcommand{\PHI}{\ensuremath\mathrm\Phi}

\newcommand{\eps}{\varepsilon}
\newcommand{\supp}{\textnormal{supp }}
\newcommand{\ee}{\textnormal{e}}

\newcommand{\BR}{{B_R(0)}}
\newcommand{\Br}{{B_r(0)}}
\newcommand{\Bro}{{B_{r_0}(0)}}
\newcommand{\BRZ}{{B_{R_Z}(0)}}

\newcommand{\ddt}{\frac{\mathrm d}{\mathrm dt}}

\newcommand{\ddtau}{\frac{\mathrm d}{\mathrm d\tau}}

\newcommand{\dtau}{\;\mathrm d\tau}
\newcommand{\dz}{\;\mathrm dz}

\newcommand{\dy}{\;\mathrm dy}
\newcommand{\dv}{\;\mathrm dv}

\newcommand{\ds}{\;\mathrm ds}
\newcommand{\dtx}{\;\mathrm d(t,x)}

\newcommand{\dtxv}{\;\mathrm d(t,x,v)}
\newcommand{\dyw}{\;\mathrm d(y,w)}

\newcommand{\B}{{\bar B}}

\newcommand{\delzi}{\partial_{z_i}}
\newcommand{\delzj}{\partial_{z_j}}
\newcommand{\delxi}{\partial_{x_i}}
\newcommand{\delxj}{\partial_{x_j}}

\newcommand{\delvi}{\partial_{v_i}}

\newcommand{\delt}{\partial_{t}}
\newcommand{\delx}{\partial_{x}}
\newcommand{\delv}{\partial_{v}}
\newcommand{\delz}{\partial_{z}}

\newcommand{\stimes}{{\hspace{-0.01cm}\times\hspace{-0.01cm}}}
\newcommand{\scdot}{{\hspace{1pt}\cdot\hspace{1pt}}}

\newcommand{\laplace}{\Delta}

\newcommand{\bigvert}{\;\big\vert\;}
\newcommand{\Bigvert}{\;\Big\vert\;}

\newcommand{\lbr}{\left\{ }
\newcommand{\rbr}{\right\} }
\newcommand{\conf}{\hspace{0.1cm}\widehat{=}\hspace{0.1cm}}

\newcommand{\MA}{\mathbf A}
\newcommand{\MB}{\mathbf B}
\newcommand{\MC}{\mathbf C}

\newcommand{\Ma}{\mathbf a}
\newcommand{\Mb}{\mathbf b}

\newcommand{\Mf}{\mathring{\mathbf f}}

\newcommand{\Mchi}{\boldsymbol \chi}

\newcommand{\tand}{\quad\text{and}\quad}
\newcommand{\twith}{\quad\text{with}\quad}

\newcommand{\tf}{{\tilde f}}
\newcommand{\tg}{{\tilde g}}

\newcommand{\tZ}{{\tilde Z}}
\newcommand{\tB}{{\tilde B}}

\newcommand{\interior}{{\textnormal{int\,}}}

\newcommand{\Lag}{{\mathcal{L}}}

\newcommand{\wto}{\rightharpoonup}

\newcommand{\itema}{\item[\textnormal{(a)}]}
\newcommand{\itemb}{\item[\textnormal{(b)}]}

\newcommand{\Underset}[3][0pt]{\ensuremath{\underset{\raise#1\hbox{\small\ensuremath{#2}}}{#3}}}

\begin{document}

\begin{center}	
	\LARGE{\bfseries Optimal control of a Vlasov-Poisson plasma by an external magnetic field}\\[8mm]
	\normalsize{Patrik Knopf}\\[2mm]
	\textit{University of Regensburg, 93040 Regensburg, Bavaria, Germany}\\[2mm]
	\texttt{Patrik.Knopf@mathematik.uni-regensburg.de}\\[-3mm]
	
	\begin{minipage}[h]{0.275\textwidth}
		\begin{flushright}
			\vspace{-2pt}
			\includegraphics[scale=0.05]{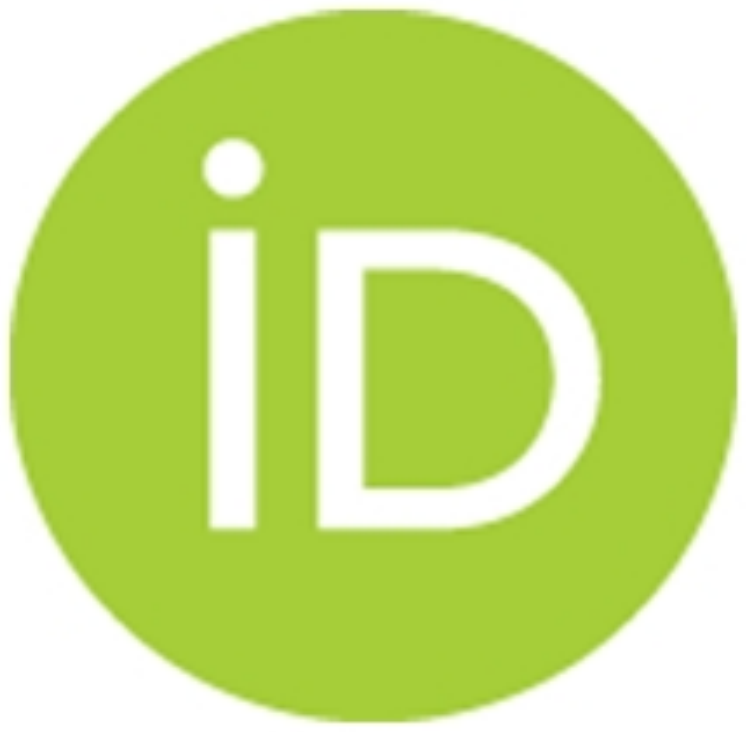} 
		\end{flushright}
	\end{minipage}
	\begin{minipage}[h]{0.5\textwidth}
		\hspace{-12pt}
		\href{https://orcid.org/0000-0003-4115-4885}{orcid.org/0000-0003-4115-4885}
	\end{minipage}

	\bigskip
	\textit{Please cite as:}  P. Knopf, Calc. Var. (2018) 57:134. \\ \url{https://doi.org/10.1007/s00526-018-1407-x}
	\bigskip
\end{center}

\begin{abstract}
	The aim of various technical applications (for example fusion research) is to control a plasma by magnetic fields in a desired fashion. In our model the plasma is described by the Vlasov-Poisson system that is equipped with an external magnetic field. We will prove that this model satisfies some basic properties that are necessary for calculus of variations. After that, we will analyze an optimal control problem with a tracking type cost functional with respect to the following topics: Necessary conditions of first order for local optimality, derivation of an optimality system, sufficient conditions of second order for local optimality, uniqueness of the optimal control under certain conditions.\pskip

	
\end{abstract}

\section{Introduction}

The three dimensional Vlasov-Poisson system in the plasma physical case is given by the following system of partial differential equations:
\begin{equation}
\label{VP}
\begin{cases}
\partial_t f + v\cdot \partial_x f - \partial_x \psi \cdot \partial_v f = 0,\\[0.25cm]
- \Delta \psi = 4\pi \rho, \quad \lim_{|x|\to\infty} \psi(t,x) = 0,\\[0.25cm]
\rho(t,x) = \int f(t,x,v)\ \mathrm dv.
\end{cases}
\end{equation}
Here $f=f(t,x,v)\ge 0$ denotes the distribution function of the particle ensemble that is a scalar function representing the density in phase space. Its time evolution is described by the first line of \eqref{VP} which is a first order partial differential equation that is referred to as the Vlasov equation. For any measurable set $M\subset \mathbb R^6$,
$\int_M f(t,x,v)\, \mathrm d(x,v)$
yields the charge of the particles that have space coordinates $x\in\mathbb R^3$ and velocity coordinates $v\in\mathbb R^3$ with $(x,v)\in M$ at time $t\ge 0$. The function $\psi$ is the electrostatic potential that is induced by the charge of the particles. It is given by Poisson's equation $-\laplace \psi = 4\pi\rho$ with an homogeneous boundary condition where $\rho$ denotes the volume charge density. The self-consistent electric field is then given by $-\partial_x \psi$. Note that both $\psi$ and $-\partial_x \psi$ depend linearly on $f$. Hence the Vlasov-Poisson system is nonlinear due to the bilinear term $-\partial_x \psi \cdot \partial_v f$ in the Vlasov equation. Assuming $f$ to be sufficiently regular (e.g., $f(t) := f(t,\cdot,\cdot)\in C^1_c (\mathbb R^6)$ for all $t \ge 0$), we can solve Poisson's equation explicitly and obtain
\begin{equation}
\label{PSIF}
\psi_f(t,x) = \iint \frac{f(t,y,w)}{|x-y|} \;\mathrm dw \mathrm dy\ \text{  for } t\ge 0, x\in\mathbb R^3.
\end{equation}
Considering $f\mapsto \psi_f$ as a linear operator we can formally rewrite the Vlasov-Poisson system as
\begin{equation}
\label{VP2}
\partial_t f + v\cdot \partial_x f - \partial_x \psi_f \cdot \partial_v f = 0.
\end{equation}
Combined with the condition
\begin{equation}
\label{IC}
f|_{t=0} = \mathring f
\end{equation}
for some function $\mathring f\in C^1_c(\RR^6)$ we obtain an initial value problem. A first local existence and uniqueness result to this initial value problem was proved by R.~Kurth \cite{kurth}. Later J. Batt \cite{batt} established a continuation criterion which claims that a local solution can be extended as long as its velocity support is under control. Finally, two different proofs for global existence of classical solutions were established independently and almost simultaneously, one by K. Pfaffelmoser \cite{pfaffelmoser} and one by P.-L. Lions and B. Perthame \cite{lions-perthame}. Later, a greatly simplified version of Pfaffelmoser's proof was published by J. Schaeffer \cite{schaeffer}. This means that the follwing result is established: Any nonnegative initial datum \linebreak$\mathring f \in C^1_c(\mathbb R^6)$ launches a global classical solution $f\in C^1([0,\infty[\times\mathbb R^6)$ of the Vlasov-Poisson system \eqref{VP} satisfying the initial condition \eqref{IC}. Moreover, for every time $t\in[0,\infty[$, $f(t)=f(t,\cdot,\cdot)$ is compactly supported in $\mathbb R^6$.
Hence equation \eqref{PSIF} and the reformulation of the Vlasov-Poisson system \eqref{VP2} are well-defined if $\mathring f \in C_c^1(\mathbb R^6)$. For more information we recommend to consider the article \cite{rein} by G. Rein that gives an overview on the most important results.\pskip

To control the distribution function $f$ we will add an external magnetic field $B$ to the Vlasov equation:
\begin{equation}
\label{VPC}
\partial_t f + v\cdot \partial_x f - \partial_x \psi_f \cdot \partial_v f + (v\times B) \cdot \partial_v f = 0, \quad f|_{t=0} = \mathring f\;.
\end{equation}
The cross product $v\times B$ occurs since, unlike the electric field, the magnetic field interacts with the particles via Lorentz force.  If we want to discuss an optimal control problem where the PDE-constraint is given by \eqref{VPC} we must firstly establish the basics for variational calculus:
\begin{itemize}
	\item We will introduce a set $\BB$ such that any field $B\in\BB$ induces a \textit{unique} and \textit{sufficiently regular strong solution} $f=f_B$ of the initial value problem~\eqref{VPC} that exists on any given time interval $[0,T]$. The set $\BB$ will be referred to as the \textit{set of admissible fields}.
	\item We will show that the solution $f_B$ depends Lipschitz-continuously on the field $B$ while the partial derivatives $\delt f_B$, $\delx f_B$ and $\delv f_B$ depend Hölder-continuously on $B$.
	\item We will prove that the operator $B\mapsto f_B$ is \textit{compact/weakly compact} and \textit{Fréchet differentiable} in some suitable sense.
\end{itemize}
With this foundations we can discuss a model problem of optimal control: Let $\mathring f \in C^2_c(\RR^6)$ be any given initial datum and let $T>0$ be any given final time. The aim is to control the distribution function $f_B$ in such a way that its value at time $T$ matches a desired distribution $f_d$ as closely as possible. We will consider the following optimal control problem
\begin{align*}
\text{Minimize}\; J(B) := \frac 1 2 \|f_B(T) - f_d\|_{L^2(\RR^6)}^2 + \frac \lambda 2 \|D_x B\|_{L^2(]0,T[\times\RR^3)}, \;\, \text{s.t.}\; B\in\BB.
\end{align*}
It will be analyzed with respect to the following topics:
\begin{itemize}
	\item Existence of a globally optimal solution,
	\item necessary conditions of first order for locally optimal solutions,
	\item derivation of an optimality system,
	\item sufficient conditions of second order for locally optimal solutions,
	\item uniqueness of the optimal control for small values of $\tfrac T \lambda$.
\end{itemize}

\section{Notation and preliminaries}

Our notation is mostly standard or self-explaining. However, to avoid misunderstandings, we fix some of it here. We will also present some basic results that are necessary for the later approach.\pskip

Let $d\in\NN$, $U\subset \RR^d$ be any open subset, $k\in \NN$ and $1\le p\le \infty$ be arbitrary. $C^k(U)$ denotes the space of $k$ times continuously differentiable functions on $U$, $C_c(U)$ denotes the space of $C^k(U)$-functions having compact support in $U$ and $C^k_b(U)$ denotes the space of $C^k(U)$-functions that are bounded with respect to the norm
\begin{align*}
\|u\|_{C^k_b(U)} := \underset{|\alpha|\le k}{\sup} \; \|D^\alpha u\|_\infty = \underset{|\alpha|\le k}{\sup} \; \underset{x\in U}{\sup}\; |D^\alpha u(x)| , \quad u\in C^k(U).
\end{align*}
For any $\gamma \in ]0,1]$, $C^{k,\gamma}(U)$ denotes the space of Hölder-continuous $C^k(U)$-functions, i.e.,
\begin{align*}
C^{k,\gamma}(U) := \big\{ u\in C^k(U) \bigvert \|u\|_{C^{k,\gamma}(U)} < \infty \big\}
\end{align*}
where for any $u\in C^k(U)$,
\begin{align*}
\|u\|_{C^k_b(U)} := \underset{|\alpha|\le k}{\sup} \big\{ \|D^\alpha u\|_\infty, [D^\alpha u]_\gamma \big\}, \; [D^\alpha u]_\gamma:= \underset{x\neq y}{\sup} \frac{|D^\alpha u(x) - D^\alpha u(y)|}{|x-y|^\gamma}.
\end{align*}
Note that $\big( C^k_b(U), \|\cdot\|_{C^k_b(U)} \big)$ and $\big( C^{k,\gamma}(U), \|\cdot\|_{C^{k,\gamma}(U)} \big)$ are Banach spaces. $L^p(U)$ denotes the standard $L^p$-space on $U$ and $W^{k,p}(U)$ denotes the standard Sobolov space on $U$ as, for instance, defined by E. Lieb and M. Loss in \cite[s.\,2.1,6.7]{lieb-loss}. If $U=\RR^d$ we will sometimes omit the argument "$(\RR^d)$". For example, we will just write $L^p$, $W^{k,p}$ or $C^k_b$ instead of $L^p(\RR^d)$, $W^{k,p}(\RR^d)$ or $C^k_b(\RR^d)$. If $U\neq \RR^d$ we will not use this abbreviation. We will also use Banach space-valued Sobolev spaces as defined by L. C. Evans \cite[p.\,301-305]{evans}. For any Banach space $X$, $L^p(0,T;X)$ denotes the space of Banach space-valued $L^p$-functions $[0,T]\ni t\mapsto u(t) \in X$. Analogously, $W^{k,p}(0,T;X)$ denotes the Banach space-valued Sobolev space. The following properties are essential:
\begin{lemma}
	\label{SOBBOCH}
	Let $T>0$, $d,m \in\NN$, $U\subset \RR^d$ be any open subset and $1\le p,q< \infty$ be arbitrary. Then the following holds:
	\begin{itemize}
		\item[\textnormal{(a)}] For any function $u\in L^p(0,T;W^{k,q}(\RR^d))$ there exists some sequence \linebreak$(u_j)\subset C^\infty(]0,T[\times \RR^d)$ such that
		\begin{gather*}
		\forall j\in\NN\; \exists r_j>0\; \forall t\in[0,T]: \supp u_j(t) \subset B_{r_j}(0)\\
		\text{and}\quad u_j \to u \quad \text{in}\; L^p(0,T;W^{k,q}(\RR^d)).
		\end{gather*}
		\item[\textnormal{(b)}] For any function $u\in W^{k,p}(0,T;L^q(\RR^d))$ there exists some sequence \linebreak$(u_j)\subset C^\infty(]0,T[\times \RR^d)$ such that
		\begin{gather*}
		\forall j\in\NN\; \exists r_j>0\; \forall t\in[0,T]: \supp u_j(t) \subset B_{r_j}(0)\\
		\text{and}\quad u_j \to u \quad \text{in}\; W^{k,p}(0,T;L^q(\RR^d)).
		\end{gather*}
		\item[\textnormal{(c)}] $L^p(]0,T[\times\RR^d)= L^p\big(0,T;L^p(\RR^d)\big)$.
		\item[\textnormal{(d)}] $W^{1,p}(]0,T[\times\RR^d) = W^{1,p}\big(0,T;L^p(\RR^d)\big) \cap L^p\big(0,T;W^{1,p}(\RR^d)\big)$.
	\end{itemize}
\end{lemma}
\smallskip
As these results are very common, we will not give a proof of this lemma in this paper. For more information on this topic we recommend to consider \cite{yosida}, \cite{pettis} and \cite{kreuter}.\\[-2mm]

In order to write down the three dimensional Vlasov-Poisson system concisely we will first define some operators and notations: For $d\in\NN$, $1\le p\le \infty$ and $r>0$ let $L^p_r(\RR^d)$ denote the set of functions $ \varphi\in L^p(\RR^d)$ having compact support $\supp \varphi\subset B_r(0)\subset\RR^d$. Then the operator
\begin{align}
\rho.\colon L^2_r(\RR^6) \to L^2(\RR^3),\;  \varphi \mapsto \rho_\varphi\quad\text{with}\quad \rho_\varphi(x):=\int  \varphi(x,v) \dv,\; x\in\RR^3
\end{align}
is linear and bounded. It also holds that $\rho_\varphi\in L^2_r(\RR^3)$ for any $ \varphi\in L^2_r(\RR^6)$. Let now $R>0$ be any arbitrary radius. From the Calderon-Zygmund inequality \linebreak\cite[p.\,230]{gilbarg-trudinger} we can conclude that
\begin{align}
\label{DEFPSI}
\psi.\colon L^2_r(\RR^6) \to H^2\big(\BR\big),\;  \varphi \mapsto \psi_\varphi \quad\text{with}\quad \psi_\varphi(x):= \int \frac{\rho_\varphi(y)}{|x-y|}\dy
\end{align}
is a linear and bounded operator. According to E. Lieb and M. Loss \cite[s.\,6.21]{lieb-loss}, the gradient of $\psi_\varphi$ is given by
\begin{align*}
\delx\psi_\varphi(x)= -\int \frac{x-y}{|x-y|^3}\;\rho_\varphi(y)\dy, \quad x\in \RR^3
\end{align*}
and then, because of \eqref{DEFPSI}, the operator
\begin{align}
\begin{aligned}
\delx\psi.\colon L^2_r(\RR^6) \to H^1\big(\BR;\RR^3\big),\;  \varphi \mapsto \delx\psi_\varphi \\
\end{aligned}
\end{align}
is also linear and bounded. Some more properties of the potential $\psi_\varphi$ and its field $\delx \psi_\varphi$ are given by the following lemma.

\begin{lemma}
	\label{NPOT}
	Let $\varphi \in L^2_r(\RR^6)$ be arbitrary. 
	\begin{itemize}
		\item[\textnormal{(a)}] $\psi_\varphi\in H^2_{loc}(\RR^3)$ has a continuous representative and is the unique solution of the boundary value problem
		\begin{align*}
		-\laplace \psi_\varphi = 4\pi\rho_\varphi \;\;\text{a.e.\,on}\; \RR^3 ,\quad \underset{|x|\to\infty}{\lim} \psi_\varphi = 0.
		\end{align*}
		\item[\textnormal{(b)}] Let $1<p,q<\infty$ be any real numbers and suppose that additionally $\varphi\in L^p(\RR^6)$. Then ${\rho_\varphi \in L^p(\RR^3)}$ and there exists some constant $C>0$ depending only on $p$ and $r$ such that \begin{align*}
		\|\psi_\varphi\|_{L^p} &\le C\,\|\varphi\|_{L^q}, && \text{where}\quad \tfrac 1 q = \tfrac 2 3 +\tfrac 1 p\,, \text{ i.e.},\; q=\tfrac{3p}{2p+3}	\\
		\|\delx\psi_\varphi\|_{L^p} &\le C\,\|\varphi\|_{L^q}, && \text{where}\quad \tfrac 1 q = \tfrac 1 3 +\tfrac 1 p\,, \text{ i.e.},\; q=\tfrac{3p}{p+3}\\
		\|D_x^2 \psi_\varphi\|_{L^p} &\le C\,\|\varphi\|_{L^q } , && \text{where}\quad \tfrac 1 q = \tfrac 0 3 +\tfrac 1 p\,, \text{ i.e.},\; q=p .
		\end{align*}
		\item[\textnormal{(c)}] Suppose that additionally $\varphi\in L^\infty(\RR^6)$. Then $\rho_\varphi\in L^\infty(\RR^3)$ and there exists some constant $C>0$ depending only on $r$ such that $$\|\delx \psi\|_{L^\infty} \le C\,\|f\|_{L^\infty}.$$
		Moreover $\delx\psi_\varphi \in C^{0,\gamma}(\RR^3;\RR^3)$ for any $\gamma\in ]0,1[$ and there exists some constant $C>0$ depending only on $r$ such that 
		$$ |\delx\psi_\varphi(x) - \delx\psi_\varphi(y)| \le C\,\|\varphi\|_{L^\infty}\,|x-y|^\gamma, \quad x,y\in\RR^3. $$
	\end{itemize}
\end{lemma}

\begin{proof}
	\textit{Item} (a): The Calderon-Zygmund lemma (\cite[p.\,230]{gilbarg-trudinger}) states that $\psi_\varphi$ is in $H^2_{loc}(\RR^3)$ and satisfies $	-\laplace \psi_\varphi = 4\pi\rho_\varphi$ almost everywhere on $\RR^3$. By Sobolev's inequality, $\psi_\varphi$ has a continuous representative and one can easily show that $\psi_\varphi(x) \to 0$ if $|x|\to \infty$, i.e., $\psi_\varphi$ satisfies the boundary value problem. Any other solution is then given by $\psi_\varphi + h$ where $\laplace h = 0$ almost everywhere and $h$ satisfies the boundary condition. Then, by Weyl's lemma, $h$ is a harmonic function and thus $h=0$ which means uniqueness. This proves (a). 
	
	\textit{Item} (b): $\rho_\varphi\in L^p(\RR^6)$ with $\|\rho_\varphi\|_{L^p} \le C\, \|\varphi\|_{L^p}$ is a direct consequence of Jensen's inequality. The first two inequalities are already established by E.~Stein~\cite[p.\,119]{stein}. Note that the Calderon-Zygmund inequality also yields $\|D^2 \psi_\varphi \|_{L^p(B_{2r}(0))} \le C\, \|\varphi\|_{L^p}$. Moreover, if $|x|\ge 2r$,
	\begin{align*}
	|\delxi \delxj \psi_\varphi(x)| \le \int\limits_{|y|<r} \frac{C}{|x-y|^3}\, |\rho_\varphi(y)| \dy \le C\, |x|^{-3}\, \|\rho_\varphi\|_{L^1} \le C\, |x|^{-3}\, \|\varphi\|_{L^p}\,.
	\end{align*}
	Thus $\|D^2 \psi_\varphi \|_{L^p(\RR^3\setminus B_{2r}(0))} \le C\, \|\varphi\|_{L^p}$ which completes the proof of (b).\\[-2mm]
	
	\textit{Item} (c): $\rho_\varphi \in L^\infty(\RR^6)$ is obvious. It holds that
	\begin{align*}
	|\delx\psi_\varphi(x)| &\le \|\rho_\varphi\|_{L^\infty} \int\limits_{|y|<r} |x-y|^{-2} \dy \le C\,\|\varphi\|_{L^\infty} \int\limits_{|y|<3r} |y|^{-2} \dy, && |x|\le 2r,\\
	|\delx\psi_\varphi(x)| &\le \int\limits_{|y|<r} \frac{|\rho_\varphi(y)|}{|x-y|^2} \dy \le C\, |x|^{-2}\, \|\rho_\varphi\|_{L^1} \le C\, |x|^{-2}\, \|\varphi\|_{L^\infty}, && |x|> 2r
	\end{align*}
	and then $\|\delx\psi_\varphi\|_{L^\infty} \le C\,\|\varphi\|_{L^\infty}$ immediately follows. The second assertion is established by E. Lieb and M. Loss \cite[s.\,10.2]{lieb-loss}. 
\qed\end{proof}

\smallskip

We will also use the notation $\rho_f$, $\psi_f$ and $\delx\psi_f$ for functions $f=f(t,x,v)$ with $t\ge 0$, $x,v\in\RR^3$. In this case we will write
\begin{align*}
\rho_f(t,x) = \rho_{f(t)}(x),\qquad
\psi_f(t,x) = \psi_{f(t)}(x),\qquad
\delx\psi_f(t,x) = \delx\psi_{f(t)}(x)
\end{align*}
for any $t$ and $x$. As already mentioned in the introduction we consider the following initial value problem:\\
\begin{equation}
\label{VPSU}
\begin{cases}
\partial_t f + v\cdot \partial_x f - \partial_x \psi_f \cdot \partial_v f + (v\times B) \cdot \partial_v f= 0 & \text{ on }[0,T]\times \RR^6\;,\\[0.25cm]
f|_{t=0} = \mathring f & \text{ on }\RR^6\,.
\end{cases}
\end{equation}
In the following let $T>0$ and $\mathring f\in C^2_c(\RR^6;\RR_0^+)$ be arbitrary but fixed. Let $B=B(t,x)$ be a given external magnetic field and let $f=f(t,x,v)$ denote the distribution function that is supposed to be controlled. Its electric field $\delx\psi_f=\delx\psi_f(t,x)$ is formally defined as described above. In the following we will show that the solution $f$ satisfies the required condition "$f(t)\in L^2_r(\RR^6)$" that ensures that $\rho_f$, $\psi_f$ and $\delx\psi_f$ are well-defined. Of course this is possible only if the magnetic field $B$ is regular enough. The regularity of those fields will be specified in the following section.

\section{Admissible fields and the field-state operator}

\subsection{The set of admissible fields}

We will now introduce the set our magnetic fields will belong to: The set of admissible fields. For $T>0$ and $\beta>3$ let $\WW=\WW(\beta)$ denote the reflexive Banach space $L^2\big(0,T;W^{2,\beta}(\RR^3;\RR^3)\big)$, let $\HH$ denote the Hilbert space $L^2\big(0,T;H^1(\RR^3;\RR^3)\big)$ and let $\|\cdot\|_\WW$ and $\|\cdot\|_\HH$ denote their standard norms. Then $\VV:=\WW \cap \HH$ with $\|\cdot\|_\VV := \|\cdot\|_\WW + \|\cdot\|_\HH$ is also a Banach space.

%
%

\begin{definition}
	\label{DAC}
	Let $K>0$ and $3<\beta<\infty$ be arbitrary fixed constants. Then
	\begin{align*}
	\BB := \lbr B\in \VV \Bigvert \|B\|_\VV \le K \rbr
	\end{align*}
	is called the \textbf{set of admissible fields}.
\end{definition}

\begin{remark}
	In the approach of Section 3 and 4 it would be sufficient to consider fields $B\in\WW$ with $\|B\|_\WW \le K$. However, in the model that is discussed in Sections 5-7 the regularity $B\in\VV$ will be necessary, especially for Fréchet differentiability of the cost functional. Therefore we will use this condition right from the beginning.
\end{remark}

The most important properties of the set of admissible fields are listed in the following lemma.

%
%

\begin{lemma}
	\label{LAC}
	The set of admissible fields $\BB$ has the following properties:
	\begin{itemize}
		\item[\textnormal{(a)}]
		$\BB$ is a bounded, convex and closed subset of $\VV$.
		\item[\textnormal{(b)}]
		The space $W^{j,\beta}(\RR^3;\RR^3)$ is continuously embedded in $C^{j-1,\gamma}(\RR^3;\RR^3)$ for $j\in\NN$ and $\gamma=\gamma(\beta)=1-\tfrac{3}{\beta}$. Thus there exist constants $k_0,k_1>0$ depending only on $\beta$ such that for all $B\in\BB$,
		$$ \|B(t)\|_{C^{0,\gamma}} \le k_0\; \|B(t)\|_{W^{1,\beta}}, \quad \|B(t)\|_{C^{1,\gamma}} \le k_1\; \|B(t)\|_{W^{2,\beta}}$$
		for almost all $t\in[0,T]$. Moreover for any $r>0$ there exist constants $k_2,k_3>0$ depending only on $\beta$ and $r$ such that for all $B\in\BB$,
		\begin{align*}
		\|B(t)\|_{C^{0,\gamma}(B_r(0))} &\le k_2\; \|B(t)\|_{W^{1,\beta}(B_r(0))}, \\ \|B(t)\|_{C^{1,\gamma}(B_r(0))} &\le k_3\; \|B(t)\|_{W^{2,\beta}(B_r(0))}
		\end{align*}
		for almost all $t\in[0,T]$. 	
		\item[\textnormal{(c)}] The space $\WW$ is continuously embedded in $L^2(0,T;C^{1,\gamma})$. Thus for all \linebreak $B\in\BB$ it holds that $B\in L^2(0,T;C^{1,\gamma})$ with $ \|B\|_{L^2(0,T;C^{1,\gamma})} \le k_1 K \;. $
		\item[\textnormal{(d)}]
		Let $\MM$ denote the set
		\begin{align*}
		\left\{ B \in C^\infty([0,T]\times\RR^3;\RR^3)) \left|
		\begin{aligned}
		&\|B\|_\WW \le 2 K \;\text{and}\;\; \exists m>0\; \forall \in[0,T] : \\
		&  \supp B(t) \subset B_m(0)\subset \RR^3 
		\end{aligned}
		\right. \right\}.
		\end{align*}
		Then for any $B\in \BB$, there exists a sequence $(B_k)_{k\in\NN}\subset\MM$ such that
		$$\|B-B_k\|_\WW\to 0,\quad k\to\infty\;.$$
		\item[\textnormal{(e)}]
		$\BB\subset \VV$ is weakly compact, i.e., any sequence in $\BB$ contains a subsequence converging weakly in $\VV$ to some limit in $\BB$.
	\end{itemize}
\end{lemma}

\begin{proof} 
	(a) is obvious and (b) is a direct consequence of Sobolev's embedding theorem and the fact that for all $B\in\BB$, $B(t)\in W^{2,\beta}(\RR^3;\RR^3)$ for almost all $t\in[0,T]$. Then the $k_1$-inequality of (b) immediately implies (c). (d) follows instantly from Lemma \ref{SOBBOCH}(a). Without loss of generality, we can assume that $\|B\|_\WW\le 2K$. As $\BB$ is a bounded subset of $\WW$ the Banach-Alaoglu theorem implies that any sequence ${(B_k)\subset\BB}$ contains a subsequence $(B_k^*)$ converging weakly to some limit $B\in \WW$. Now, $(B_k^*)$ is a bounded sequence in $\HH$ and thus it has a subsequence $(B_k^{**})$ that converges weakly to some limit in $\HH$. Because of uniqueness, this limit must be $B$. Hence $B^{**}_k\wto B$ in $\VV$ and since the norm $\|\cdot\|_\VV$ is weakly lower semicountinuous it follows that $\|B\|_\VV \le K$. That is (e).
\qed\end{proof}

\subsection{The characteristic flow of the Vlasov equation}

Since the Vlasov equation is a first-order partial differential equation, it suggests itself to consider the characteristic system.
On that point, we will consider a general version of the Vlasov equation,
\begin{equation}
\label{VEF}
\partial_t f + v\cdot \partial_x f + F \cdot \partial_v f + v \times G \cdot \partial_v f= 0,
\end{equation}
with given fields $F=F(t,x)$ and $G=G(t,x)$. Then the following holds:

%
%

\begin{lemma}
	\label{CHS}
	Let $I\subset \mathbb R$ be an interval and let $F,G\in C(I\times\mathbb R^3;\mathbb R^3)$ be continuously differentiable with respect to $x$ and bounded on $J\times \mathbb R^3$ for every compact subinterval $J\subset I$. Then for every $t\in I$ and $z=(x,v)\in \mathbb R^6$ there exists a unique solution $I\ni s\mapsto (X,V)(s,t,x,v)$ of the characteristic system
	\begin{equation}
	\dot x = v,\quad
	\dot v = F(s,x) + v\times G(s,x)
	\end{equation}
	to the initial condition $(X,V)(t,t,x,v)=(x,v)$. The characteristic flow \linebreak$Z:=(X,V)$ has the following properties:
	\begin{itemize}
		\item[\textnormal{(a)}]
		$Z:I\times I\times \mathbb R^6 \to \mathbb R^6$ is continuously differentiable.
		\item[\textnormal{(b)}]
		For all $s,t\in I$ the mapping $Z(s,t,\cdot):\mathbb R^6 \to \mathbb R^6$ is a $C^1$-diffeomorphism with inverse $Z(t,s,\cdot)$, and $Z(s,t,\cdot)$ is measure preserving, i.e.,
		\begin{equation*}
		\det \frac{\partial Z}{\partial z} (s,t,z) = 1, \quad s,t\in I,\; z\in\RR^6\;.
		\end{equation*}
	\end{itemize}
\end{lemma}

The relation between the characteristic flow and the solution $f$ of the Vlasov equation \eqref{VEF} is described by the following lemma.

%
%

\begin{lemma}
	\label{FL2}
	Under the assumptions of Lemma \ref{CHS} the following holds:
	\begin{itemize}
		\item[\textnormal{(a)}]
		A function $f\in C^1(I\times\mathbb R^6)$ satifies the Vlasov equation \eqref{VEF} iff it is constant along every solution of the characteristic system \eqref{CHS}.
		\item[\textnormal{(b)}]
		Suppose that $0\in I$. For $\mathring f\in C^1(\mathbb R^6)$ the function $f(t,z):=\mathring f(Z(0,t,z))$, $t\in I,z\in \mathbb R^6$
		is the unique solution of \eqref{VEF} in the space $C^1(I\times \mathbb R^6)$ with $f(0)=\mathring f$. If $\mathring f$ is nonnegative then so is f. For all $t\in I$ and $1\le p\le \infty$, 
		$$\textnormal{supp}f(t) = Z(t,0,\textnormal{supp} \mathring f) \tand \|f(t)\|_p = \|\mathring f\|_p\,.$$
		
	\end{itemize}
\end{lemma}

For $G=0$ the proofs of both lemmata are presented by G. Rein \cite[p.\,394]{rein}. The proofs for a general field $G\in C(I\times\mathbb R^3;\mathbb R^3)$ proceed analogously.

\subsection{Classical solutions for smooth external fields}

As already mentioned in the introduction, the standard initial value problem $(\eqref{VP2}, \eqref{IC})$ posesses a unique global classical solution. This result holds true if the Vlasov equation is equipped with an external magnetic field $B$ in $C\big([0,T];C^1_b(\RR^3;\RR^3)\big)$ which will be established in the next theorem. Unfortunately the proof does not work if the field is merely an element of $\BB$. Since such fields are only $L^2$ in time, the same holds for the right-hand side of the characteristic system. This makes it impossible to determine a solution in the classical sense of ordinary differential equations. However, we can approximate any field $B\in \BB$ by a sequence $(B_k)_{k\in\NN}\subset\MM \subset C\big([0,T];C^1_b(\RR^3;\RR^3)\big)$ according to Lemma \ref{LAC}. This allows us to construct a certain kind of strong solution to the field $B$ as a limit of the classical solutions induced by the fields $B_k$.

%
%

\begin{theorem}
	\label{GCS}
	Let $B \in C\big([0,T];C^1_b(\RR^3;\RR^3)\big)$ with $\|B\|_\WW\le 2K$ be arbitrary. Then the initial value problem \eqref{VPSU} possesses a unique classical solution \linebreak ${f\in C^1([0,T]\times\RR^6)}$. Moreover for all $t\in[0,T]$, ${f(t)=f(t,\cdot,\cdot)}$ is compactly supported in $\mathbb R^6$ in such a way that there exists some constant $R>0$ depending only on $T$, $\mathring f$, $K$ and $\beta$ such that for all $t\in[0,T]$,
	\begin{equation*}
	\textnormal{supp } f(t) \subset B_R^6(0) = \left\{(x,v)\in\mathbb R^6 : |(x,v)| < R\right\}.
	\end{equation*}
	If $B \in C\big([0,T];C^2_b(\RR^3;\RR^3)\big)$, then additionally  $f\in C\big([0,T];C^2_b(\RR^6)\big)$.
\end{theorem}

\smallskip

\begin{proof} \textit{Step 1 - Local existence and uniqueness:} For the standard Vlasov-Poisson system $(\eqref{VP2}, \eqref{IC})$ the existence and uniqueness of a local classical solution was firstly established by R. Kurth \cite{kurth}. As the field ${B\in C\big([0,T];C^1_b(\RR^3;\RR^3)\big)}$ is regular enough the existence and uniqueness of a local classical solution to our problem can be proved analogously. In this paper we will only sketch the most important steps of that proof. The idea is to define a recursive sequence by
	\begin{align*}
	f_0(t,z) := \mathring f(z) \tand f_{k+1}(t,z) := \mathring f \big(Z_k(0,t,z) \big),\quad k\in \NN_0
	\end{align*}
	for any $t\ge 0$ and $z=(x,v)\in\RR^6$ where $Z_k$ denotes the solution of
	\begin{align*}
	\dot z = {\begin{pmatrix} \dot x \\ \dot v \end{pmatrix}} = \begin{pmatrix} v \\ -\delx\psi_{f_k}(s,x) + v\times B(s,x) \end{pmatrix} \twith Z_k(t,t,z) = z\;.
	\end{align*}
	By induction we obtain that for any $k\in\NN_0$, $Z_k$ is continuously differentiable with respect to all its variables and $f_k\in C^1([0,T]\times\RR^6)$. Moreover, according to Lemma \ref{FL2}, it holds that $f_{k+1} \in C^1([0,T]\times\RR^6)$ is the unique solution of the initial value problem
	\begin{align*}
	\delt f + v\cdot\delx f - \delx\psi_{f_k}\cdot\delv f +(v\times B)\cdot\delv f = 0, \quad f\big\vert_{t=0} = \mathring f\,.
	\end{align*}
	We intend to prove that the sequence $(f_k)$ converges to the solution of the initial value problem \eqref{VPSU} if $k$ tends to infinity. Analogously to Kurth's proof we can show that there exists $\delta>0$ and functions $Z$, $f$ with $Z\in C([0,\delta_0]^2\times\RR^6)$, $f\in C([0,\delta_0]\times\RR^6)$ for any $\delta_0<\delta$ such that
	\begin{align*}
	Z(s,t,z) = \underset{k\to\infty}{\lim} Z_k(s,t,z) \tand f(t,z) = \mathring f\big(Z(0,t,z)\big) = \underset{k\to\infty}{\lim} f_k(t,z)
	\end{align*}
	uniformely in $s$, $t$ and $z$. For any arbitrary $\delta_0<\delta$ it turns out that $(\delx\psi_{f_k})$ and $(D_x^2\psi_{f_k})$ are Cauchy sequences in $C_b([0,\delta_0]\times\RR^3)$. This implies that $\delx\psi_f$ and $D_x^2\psi_f$ lie in $C_b([0,\delta_0]\times\RR^3)$ and consequently ${Z\in C^1([0,\delta_0]^2\times \RR^6)}$. As $\delta_0$ was arbitrary this yields ${f\in C^1([0,\delta[\times \RR^6)}$. Thus $f$ is a local solution of the initial value problem \eqref{VPSU} on the time interval $[0,\delta[$ according to Lemma \ref{FL2} as it is constant along any characteristic curve. \bpskip
	
	\textit{Step 2 - Higher regularity:} If $B \in C([0,T];C^2_b)$ it additionally follows by induction that for all $k\in \NN_0$ and $s,t\in[0,T]$, 
	\begin{gather*}
	\delx\psi_{f_k}\in C([0,T];C^2_b(\RR^3;\RR^3)), \quad Z_k(s,t,\cdot) \in C^2(\RR^6), \quad f_k \in C([0,T];C^2_b(\RR^6)).
	\end{gather*}
	Moreover, if $\delta$ is sufficiently small, one can also show that $(D_x^3 \psi_{f_k})$ is a Cauchy sequence in $C_b([0,\delta_0]\times\RR^3)$ for any $\delta_0<\delta$. Thus we can conclude that 
	\begin{gather*}
	\delx\psi_{f}\in C([0,T];C^2_b(\RR^3;\RR^3)), \quad Z(s,t,\cdot) \in C^2(\RR^6), \quad f \in C([0,T];C^2_b(\RR^6)).\pskip
	\end{gather*}
	
	\textit{Step 3 - Continuation onto the interval $[0,T]$:} Obviously Batt's continuation criterion (cf.\;J. Batt \cite{batt}) also holds true in our case. This means that we can show that the solution exists on $[0,T]$ by the following argumentation:
	We assume that $[0,T^*[$ with $T^*\le T$ is the right maximal time interval of the local solution and we show that
	\begin{align}
	\begin{aligned}
	P(t) :&= \max \{ |v| : (x,v)\in \text{supp } f(s),0\le s\le t\}\\[2mm]
	& = \max \{ |V(s,0,x,v)| : (x,v)\in \text{supp } \mathring f,0\le s\le t\}
	\end{aligned}
	\end{align}
	is bounded on $[0,T^*[$. But then, according to Batt, the solution $f$ can be extended beyond $T^*$ which is a contradiction as $T^*$ was chosen to be maximal. Hence we can conclude that the solution exists on the whole time interval $[0,T]$. \pskip
	
	For the standard Vlasov-Poisson system (without an external field) such a bound on $P(t)$ is established in the Pfaffelmoser-Schaeffer proof \cite{pfaffelmoser,schaeffer}. We will proceed analogously and single out one particle in our distribution. Mathematically, this means to fix a characteristic $(X,V)(s)=(X,V)(s,0,x,v)$ with $(X,V)(0)=(x,v)\in\text{supp }\mathring f$. Now suppose that $0\le\delta\le t< T^*$. In the following, constants denoted by $C$ may depend only on $\mathring f$, $T$, $K$ and $\beta$. The aim in the Pfaffelmoser-Schaeffer proof is to bound the difference ${|V(t)-V(t-\delta)|}$ from above by an expression in the shape of $C\delta P(t)^\alpha$ where $\alpha < 1$ is essential. In our case an analogous approach would merely yield some bound that is ideally in the fashion of $ C\delta P(t)^\alpha + C\sqrt{\delta}P(t)$ because of the additional field term in the $\dot v$ equation of the characteristic system. However, we can use the fact that an external magnetic field changes only the direction of a particle's velocity vector but not its modulus. This is reflected in the following computation: For $s\in [t-\delta,t]$,
	\begin{align*}
	|V(s)|^2 & = |V(t-\delta)|^2 + \int\limits_{t-\delta}^s \ddtau |V(\tau)|^2 \dtau \\[0.15cm]
	& \le P(t-\delta)^2 + 2\int\limits_{t-\delta}^s \big|\partial_x \psi_f(\tau,X(\tau))\big||V(\tau)| \dtau \;.
	\end{align*}
	The quadratic version of Gronwall's lemma (cf. Dragomir \cite[p.\,4]{dragomir}) and the definition of $\delx\psi_f$ then impliy that
	\begin{align*}
	|V(t)| 
	\le P(t-\delta) + \int\limits_{t-\delta}^t \iint \frac{f(s,y,w)}{|y-X(s)|^2}\ \mathrm dw\mathrm dy \ds \;.
	\end{align*}
	Using this inequality, the rest of the proof proceeds very similarly to the\linebreak Pfaffelmoser-Schaeffer proof.
\qed\end{proof}

Temporarily we will write $f_B$ to denote the classical solution that is induced by the field~$B$. In order to prove that any field $B\in\BB$ still induces a strong solution of the initial value problem the following two lemmata are essential. For fields $B\in\MM$ Lemma \ref{LIP} asserts that $f_B$ depends Lipschitz continuously on $B$ while its derivatives $\delz f_B$ and $\delt f_B$ are Hölder continuous with respect to $B$. In the course of the construction of a strong solution to some field $B\in\BB$ we will approximate $B$ by a sequence $(B_k)\subset\MM$ and then Lemma \ref{LIP} will ensure that $(f_{B_k})$, $(\delt f_{B_k})$ and $(\delz f_{B_k})$ are Cauchy sequences in some sense. To prove Lemma \ref{LIP} we will need some uniform bounds that are established in Lemma \ref{ES1}.

%
%

\begin{lemma}
\label{ES1}
Let $B\in \MM$ be any field. For $t,s\in[0,T]$ and $z=(x,v)\in\RR^6$ let $Z_B=Z_B(s,t,z)$ $=(X_B,V_B)(s,t,x,v)$ be the solution of the characteristic system with $Z_B(t,t,z)=z$. Furthermore let $f_B$ be the classical solution of the initial value problem \eqref{VPSU} to the field $B$. Then, there exist constants $R_Z\ge R$, $c_1,c_2,c_3,c_4>0$ depending only on $\mathring f$, $T$, $K$, and $\beta$ such that for all $t,s\in [0,T]$,\\[-0.3cm]
\begin{gather*}
\|Z_B(s,t,\cdot)\|_{L^\infty(B^6_R(0))} \le R_Z\,,\quad
\|D_z Z_B(s,t,\cdot)\|_{L^\infty(B^6_R(0))}\le c_1\,,\\
\|\delz f_B(t)\|_\infty \le c_2\,,\quad
\|D_x^2 \psi_{f_B}(t)\|_\infty \le c_3\,,\quad
\|\delt f_B\|_{L^2(0,T;C_b)}\le c_4 \,.
\end{gather*}
$Z_B$ and $f_B$ are twice continuously differentiable with respect to $z$ and there exist constants $c_5,c_6>0$ depending only on $\mathring f$, $T$, $K$, and $\beta$ such that for all $s\in[0,T]$,
\begin{gather*}
\big\|t\mapsto Z_B(s,t,\cdot)\big\|_{L^\infty(0,T;W^{2,\beta}(\BR))}\le c_5, \quad \|D_z^2 f_B\|_{L^\infty(0,T;L^\beta)} \le c_6\;.
\end{gather*}
\end{lemma}

\begin{proof}
\label{PES1}
Let $s,t\in [0,T]$ and ${z\in \BR}$ be arbitrary (without loss of generality $s\le t$) and let $i,j\in {1,...,6}$ be arbitrary indices. Let $B\in\MM$ be an arbitrary field and let $Z_B\colon [0,T] \times [0,T] \times \RR^6 \to \RR^6$ denote the induced solution of the characteristic system satisfying the initial condition $Z_B(t,t,z) = z$.
For brevity, we will use the notation $Z_B(s)=Z_B(s,t,z)$. The letter $C$ will denote a positive generic constant depending only on $\mathring f$, $K$, $T$ and $\beta$. Using the fundamental theorem of calculus we obtain that
\begin{align*}
|Z_B(s)|^2 
& \le R^2 + \int\limits_s^T |Z_B(\tau)|^2\ \mathrm d\tau + \int\limits_s^T |Z_B(\tau)|\|\partial_x \psi_{f_B}(\tau) \|_\infty\dtau.
\end{align*}
Hence applying first the standard version and then the quadratic version of Gronwall's lemma provides that
\begin{align}
\|Z_B(s)\|_{L^\infty(\BR)} \le C + C \int\limits_s^T \|\partial_x \psi_{f_B}(\tau) \|_\infty\dtau \le C  =: R_Z.
\end{align}
The partial derivatives $\delzi Z_B$, $i=1,...,6$ can be bounded by\vspace{-2mm}
\begin{align*}
|\partial_{z_i} Z_B(s)| \le 1 + \int\limits_s^t C \big( 1 + \|D_x^2 \psi_{f_B}(\tau)\|_\infty + \|B(\tau)\|_{W^{1,\infty}} \big) |\partial_{z_i} Z_B(\tau)|\;\mathrm d\tau
\end{align*}
and consequently, by Gronwall's lemma,\vspace{-2mm}
\begin{align}
\label{ADZ}
\|D_z Z_B(s)\|_{L^\infty(\BR)} 
\le C \exp\left( C \int\limits_s^t \|D_x^2 \psi_{f_B}(\tau)\|_\infty \;\mathrm d\tau \right)
\end{align}
According to G. Rein in \cite[p.\,389]{rein},
\begin{align}
\label{AD2PSI}
\|D_x^2 \psi_{f_B}(t)\|_\infty &\le C\left[ (1+\|\rho_{f_B}(t)\|_\infty)(1+ \ln_+\|\partial_x \rho_{f_B}(t)\|_\infty) + \|\rho_{f_B}(t)\|_{L^1} \right] \nonumber \\
&\le C+C\ln_+\|\partial_x \rho_{f_B}(t)\|_\infty
\end{align}
as $\|\rho_{f_B}(t)\|_\infty \le \tfrac 4 3 R^3 \pi \|\mathring f\|_\infty$ and $\|\rho_{f_B}(t)\|_{L^1} = \|f_B(t)\|_{L^1} = \|\mathring f\|_{L^1}$. Moreover,
\begin{align*}
&|\partial_x \rho_{f_B}(t,x)|  \le \int\limits_{|v|\le R} |\partial_x f(t,x,v)|\;\mathrm dv
\le \frac{4\pi}{3}R^3 \; \|\partial_z \mathring f\|_\infty \|\partial_z Z_B(0)\|_{L^\infty(\BR)} \\[-1mm]
&\quad \le \frac{4\pi}{3}R^3\; \|\partial_z \mathring f\|_\infty\; C \exp\left( C \int\limits_0^t \|\partial_x^2 \psi_{f_B}(\tau)\|_\infty\;\mathrm d\tau \right)
\end{align*}
and now, by (\ref{AD2PSI}) and Gronwall's lemma, $\|D_x^2 \psi_{f_B}(t)\|_\infty \le C =: c_3$ for all ${t\in [0,T]}$. Thus, due to \eqref{ADZ}, $\|D_z Z_B(s) \|_{L^\infty(\BR)} \le C =: c_1$ for all ${t\in [0,T]}$. 
This directly yields
\begin{align}
\label{DELZF}
\| \partial_z f_B(t) \|_\infty \le \| \partial_z \mathring f \|_\infty  \|D_z Z_B(0) \|_{L^\infty(\BR)} \le C =: c_2\;.
\end{align}
and we can finally conclude that $\| \delt f_B \|_{L^2(0,T;C_b)} \le C =: c_4$ by expressing $\delt f$ by the Vlasov equation. In Step 2 of the proof of Theorem \ref{GCS} we have already showed that for all $s,t\in [0,T]$,\vspace{-1mm}
\begin{gather*}
\delx\psi_{f_B}\in C([0,T];C^2_b(\RR^3;\RR^3)), \;\; Z_B(s,t,\cdot) \in C^2(\RR^6), \;\; f_B \in C([0,T];C^2_b(\RR^6)).
\end{gather*}
Let $i,j,k\in \{1,...,6\}$ be arbitrary. Recall that, according to Lemma \ref{NPOT},
\begin{align}
\label{ESTPSI}
\begin{aligned}
\| \delxi \delxj \partial_{x_k} \psi_{f_B} (t) \|_{L^\beta} &= \| \delxi \delxj \psi_{\partial_{x_k} f_B} (t) \|_{L^\beta} \le C\,\| \partial_{x_k} f_B (t) \|_{L^\infty}.
\end{aligned}
\end{align}
Now, for all $s,t\in[0,T]$ (without loss of generality $s\le t$),\vspace{-3mm}
\begin{align*}
&\int\limits_\BR |\delzi\delzj Z_B(s,t,z)|^\beta \dz 
	= \beta \int\limits_s^t \int\limits_\BR |\delzi\delzj Z(\tau)|^{\beta-1}\, |\delzi\delzj \dot Z_B(\tau)| \dz \dtau.
\end{align*}
Note that for all $s,t\in[0,T]$ and $z\in\BR$,\vspace{-2mm}
\begin{align*}
|\delzi\delzj \dot Z_B(\tau,t,z)| 
&\le C\, |\delzi\delzj Z_B(\tau)| \Big( 1 + \|D_x^2 \psi_{f_B}(\tau)\|_{L^\infty} + \|D_x B(\tau)\|_{L^{\infty}} \Big) \\
&\quad + C\, \Big( |D_x^3 \psi_{f_B}(\tau,X_B(\tau))| + |D_x^2 B(\tau,X_B(\tau))| \Big)\,.
\end{align*}
Now, using Hölder's inequality with exponents $p=\frac{\beta-1}{\beta}$ and $q=\beta$, estimate~\eqref{ESTPSI} and a nonlinear generalization of Gronwall's lemma (cf. \cite[p.\,11]{dragomir}) with exponent $\frac{\beta-1}{\beta}\in ]0,1[ $, we obtain that $\|\delzi\delzj Z_k(s,t,\cdot)\|_{L^\beta} \le C$. This finally implies the $c_5$-estimate and the $c_6$-estimate easily follows.
\qed\end{proof}

%
%

\begin{lemma}
\label{LIP}
Let $B,H\in \MM$ and let $f_B,f_H$ be the induced classical solutions. Moreover, let $Z_B$ denote the solution of the characteristic system to the field $B$ satisfying $Z_B(t,t,z)=z$ and let $Z_H$ be defined analogously. Then, there exist constants $\ell_1,\ell_2,L_1,L_2,L_3>0$ depending only on $\mathring f$, $T$, $K$, $\beta$ such that
\begin{align}
\label{IEQ:l1}
\|Z_B-Z_H\|_{C([0,T];C_b(\BR))} &\;\le\; \ell_1 \|B-H\|_\WW\;,\\
\label{IEQ:l2}
\|\partial_z Z_B-\partial_z Z_H\|_{C([0,T];C_b(\BR))} &\;\le\; \ell_2 \|B-H\|_\WW^{\gamma}\;,\\
\label{IEQ:L1}
\|f_B-f_H\|_{C([0,T];C_b)} &\;\le\; L_1 \|B-H\|_\WW\;,\\
\label{IEQ:L2}
\|\partial_z f_B-\partial_z f_H\|_{C([0,T];C_b)} &\;\le\; L_2 \|B-H\|_\WW^\gamma\;,\\
\label{IEQ:L3}
\|\partial_t f_B-\partial_t f_H\|_{L^2(0,T;C_b)} &\;\le\; L_3 \|B-H\|_\WW^\gamma\;
\end{align}
where $\gamma=\gamma(\beta)$ is the H\"older exponent from Lemma \ref{LAC}.
\end{lemma}

\begin{proof}
Let $B,H\in \MM$, $s,t\in[0,T]$ and $z\in\BR$ be arbitrary. Without loss of generality $s\le t$. Moreover, let $C>0$ denote a generic constant depending only on $\mathring f$, $T$ and $K$ and let $Z_B$ and $Z_H$ denote the solutions of the characteristic system to the fields $B$ and $H$ satisfying $Z_B(t,t,z) =z$ and $Z_H(t,t,z) =z$. Using the fundamental theorem of calculus, Lemma \ref{ES1} and Lemma \ref{LAC} we obtain\vspace{-2mm}
\begin{align*}
|Z_B(s) - Z_H(s)| 
&\le
C \int\limits_s^t \Big\{\big( 1 + \|D_x^2 \psi_{f_B}(\tau)\|_\infty + \|B(\tau)\|_{W^{1,\infty}} \big) |Z_B(\tau) - Z_H(\tau)| \\[-1mm]
&\hspace{-5pt}+\|\partial_x \psi_{f_B-f_H}(\tau)\|_{L^\infty(B_{R_Z}(0))}  + \|B(\tau)-H(\tau)\|_{L^\infty(B_{R_Z}(0))} \Big\} \dtau 
\end{align*}
by a straightforward computation. Thus by Lemma \ref{NPOT}\,(c) and Lemma \ref{ES1},
\begin{align*}
|Z_B(s) - Z_H(s)| &\le C \int\limits_s^t \|{f_B}(\tau)-{f_H}(\tau)\|_\infty \mathrm d\tau + C \|B-H\|_\WW
\end{align*}
and then by chain rule,
\begin{align*}
&\|f_B(t)-f_H(t)\|_\infty \le \|D\mathring f\|_\infty\; \|Z_B(0,t,\cdot) - Z_H(0,t,\cdot)\|_{L^\infty(\BR)}\\
&\quad \le C \int\limits_0^t \|{f_B}(\tau)-{f_H}(\tau)\|_\infty \dtau + C \|B-H\|_\WW
\end{align*}
which proves \eqref{IEQ:l1} and \eqref{IEQ:L1} if $\ell_1$ and $L_1$ are chosen suitably. By Lemma \ref{NPOT}\,(c),
\begin{align}
\label{HOELDPSI}
&|\delxi\delxj \psi_{f_B}(\tau,X_B(\tau)) - \delxi\delxj   \psi_{f_B}(\tau,X_H(\tau))| \notag \\
&\quad\le C\;|X_B(\tau) - X_H(\tau)|^{\gamma} \le C\;\|B-H\|_{L^2 ( 0,T;W^{1,\beta}(B_{R_Z}(0)) ) }^{\gamma}
\end{align}
for every $\tau\in[0,T]$ and every $i,j\in\{1,...,6\}$. This estimate can be used to prove \eqref{IEQ:l2} and \eqref{IEQ:L2} similarly to the above procedure. Finally \eqref{IEQ:L3} follows directly from \eqref{IEQ:L1} and \eqref{IEQ:L2} by representing $\delt f_B$ and $\delt f_H$ by their corresponding Vlasov equation.
\qed\end{proof}

\subsection{Strong solutions for admissible external fields}

Now we will show that any field $B\in\BB$ still induces a unique strong solution which can be constructed as the limit of solutions $f_{B_k}$ where $(B_k)\subset\MM$ with $B_k\to B$ in $\WW$. Such a strong solution is defined as follows:

%
%

\begin{definition}
\label{WKS}
Let $B\in\BB$ be any admissible field. We call $f$ a \textbf{strong solution} of the initial value problem \eqref{VPSU} to the field $B$, iff the following holds:
\begin{enumerate}
\item[\textnormal{(i)}]		For all $1\le p \le \infty$, $f\in W^{1,2}(0,T;L^p(\RR^6)) \cap L^2(0,T;W^{1,p}(\RR^6))$\\ $\subset C([0,T];L^p(\RR^6))$ with
$$\|f\|_{W^{1,2}(0,T;L^p)} + \|f\|_{L^2(0,T;W^{1,p})} \le C$$ for some constant $C>0$ depending only on $\mathring f$, $T$, $K$ and $\beta$.
\item[\textnormal{(ii)}]	$f$ satisfies the Vlasov equation
$$\delt  f + v\cdot \partial_x  f - \partial_x \psi_f \cdot \partial_v  f + (v\times B) \cdot \partial_v  f= 0$$
almost everywhere on $[0,T]\times\RR^6$.
\item[\textnormal{(iii)}]	$f$ satisfies the initial condition $f\big\vert_{t=0}=\mathring f$ almost everywhere on $\RR^6$,
\item[\textnormal{(iv)}]	For every $t\in[0,T]$, $\supp f(t) \subset \BR$ where $R$ is the constant from Theorem \ref{GCS}.
\end{enumerate}
\end{definition}

\smallskip

First of all one can easily establish that such a strong solution is unique.

%
%

\begin{proposition}
\label{UNIQS}
Let $B\in\BB$ be any field and suppose that there exists a strong solution $f$ of the initial value problem \eqref{VPSU} to the field $B$. Then this solution is unique.
\end{proposition}

\begin{proof} 
Suppose that there exists another strong solution $g$ to the field $B$. Then the difference $h:=f-g$ satisfies
\begin{align*}
\delt h + v\cdot\delx h -\delx\psi_h \cdot\delv f -\delx \psi_g \cdot \delv h +(v\times B)\cdot\delv h = 0\;.
\end{align*}
almost everywhere on $[0,T]\times\RR^6$. Thus by integration by parts,
\begin{align*}
&\ddt \|h(t)\|_{L^2}^2 = 2\int \delt h \; h \dz = 2 \int \delx \psi_h\cdot \delv f\; h\dz \\
&\quad \le 2 \; \|\delv f(t)\|_{\infty}\; \|\delx\psi_h(t)\|_{L^2}\; \|h(t)\|_{L^2} \le C(R)\; \|\delv f(t)\|_\infty\;  \|h(t)\|_{L^2}^2\;.
\end{align*}
As $t\mapsto \|h(t)\|_{L^2}^2$ is continuous with $\|h(0)\|_{L^2}^2 =0$, Gronwall's lemma yields $\|h(t)\|_{L^2}^2=0$ for all $t\in[0,T]$. Hence for all $t\in[0,T]$, $f(t) = g(t)$ almost everywhere on $\RR^6$ which means uniqueness.
\qed\end{proof}

Now we will show that any admissible field $B\in\BB$ actually induces a unique strong solution. Note that this solution is even more regular than it was demanded in the definition. However the weaker requirements of the definition will be essential in the later approach (see Proposition \ref{COMP}) and it will also be important that uniqueness was established under those weaker conditions.

%
%

\begin{theorem}
\label{GWS}
Let $B\in\BB$. Then there exists a unique strong solution $f$ of the initial value problem \eqref{VPSU} to the field $B$. Moreover this solution satisfies the following properties which are even stronger than the conditions that are demanded in Definition \ref{WKS}:
\begin{enumerate}
\item[\textnormal{(a)}]		$f\in W^{1,2}(0,T;C_b(\RR^6))\cap C([0,T];C^1_b(\RR^6))\cap L^\infty(0,T;W^{2,\beta}(\RR^6))$ with
\begin{gather*}
\|f(t)\|_p = \|\mathring f\|_p\;,\quad t\in [0,T]\;,\quad 1\le p\le \infty\\
\text{and}\quad
\|f\|_{W^{1,2}(0,T;C_b)} + \|f\|_{C([0,T];C^1_b)} + \|f\|_{L^\infty(0,T;W^{2,\beta})} \le C
\end{gather*}
for some constant $C>0$ depending only on $\mathring f$, $T$, $K$ and $\beta$.
\item[\textnormal{(b)}]	$f$ satisfies the initial condition $f\big\vert_{t=0}=\mathring f$ everywhere on $\RR^6$,
\end{enumerate}
\end{theorem}

\begin{proof} Let $B\in\BB$ arbitrary. According to Lemma \ref{LAC}, we can choose some sequence ${(B_k)_{k\in\NN}\subset \MM}$ with $B_k\to B$ for $k\to\infty$ in $\WW$.
Now Lemma~\ref{LIP} and Lemma~\ref{ES1} provide that for all $t\in[0,T]$ and $j,k\in\NN$,
\begin{align*}
\|f_{B_k}-f_{B_j}\|_{C([0,T];C_b)} &\le L_1 \|B_k-B_j\|_\WW\;, \\
\|\partial_z f_{B_k}-\partial_z f_{B_j}\|_{C([0,T];C_b)} &\le L_3 \|B_k-B_j\|_\WW^{\gamma}\;,\\
\|\partial_t f_{B_k}-\partial_t f_{B_j}\|_{L^2(0,T;C_b)} &\le L_4 \|B_k-B_j\|_\WW\;,\\
\|D_z^2 f_{B_k}\|_{L^\infty(0,T;L^\beta)} &\le c_6\;.
\end{align*}
where $\gamma=\gamma(\beta)$ is the constant from Lemma \ref{LAC}. Hence, $(f_{B_k})_{n\in\NN}$ is a Cauchy sequence in $C([0,T];C^1_b)\cap W^{1,2}(0,T;C_b)$. Due to completeness there exists a unique function \linebreak $f\in C([0,T];C^1_b)\cap W^{1,2}(0,T;C_b)$ such that $f_{B_k}\to f$ in this space. Since $(f_{B_k})$ is also bounded in $L^\infty(0,T;W^{2,\beta})$ by some constant depending only on $\mathring f$, $T$, $K$ and $\beta$, the Banach-Alaoglu theorem states that there exists some function ${\bar f\in L^\infty(0,T;W^{2,\beta})}$ such that $f_{B_k}\overset{*}{\rightharpoonup} \bar f$ up to a subsequence. This means that for any ${\alpha\le 2}$, the sequence $(D_z^\alpha f_{B_k})$ converges to $D_z^\alpha \bar f$ with respect to the weak-*-topology on $[L^1(0,T;L^{\beta'})]^* \conf L^\infty(0,T;L^\beta)$ where $\nicefrac{1}{\beta} +\nicefrac{1}{\beta'} =1$. Because of uniqueness of the limit it holds that $D_z^\alpha f = D_z^\alpha \bar f$ and thus
$$ f = \bar f \in W^{1,2}(0,T;C_b)\cap C([0,T];C^1_b) \cap L^\infty(0,T;W^{2,\beta}).$$
To show that $f$ is a strong solution to the field $B$, we have to verify the conditions from Definition \ref{WKS}.  The strong convergence of $(f_{B_k})$ in $W^{1,2}(0,T;C_b)$ and $C([0,T];C^1_b)$ directly implies condition (ii). Moreover,\vspace{-1mm}
\begin{align*}
\big| f(0) - \mathring f \big| = \big| f(0) - f_{B_k}(0) \big| \le C\; \| f-f_{B_k} \|_{C([0,T];C_b)} \to 0,\qquad k\to \infty.
\end{align*}
Thus $f(0)=\mathring f$ everywhere on $\RR^6$ that is (b) which directly implies (iii). Due to uniform convergence and continuity of $f$, it is evident that $f(t)$ is also compactly supported in $B_R(0)$ for every $t\in [0,T]$. That is (iv). For $1\le q\le \infty$ arbitrary, $t\in [0,T]$ and $k\in\NN$, we have
\begin{align*}
&\Big| \|f(t)\|_q - \| \mathring f \|_q \Big| = \Big| \|f(t)\|_q - \| f_{B_k}(t) \|_q \Big|\le \|f(t) - f_{B_k}(t) \|_q \\
&\; \le \big( \lambda(B_R(0)) \big)^{\frac 1 q} \|f(t) - f_{B_k}(t) \|_\infty  \le \big( 1+\lambda(B_R(0)) \big) \|f - f_{B_k} \|_{C([0,T];C_b)} \to 0,
\end{align*}
if $k\to\infty$ where $\frac 1 q \conf 0$ if $q=\infty$. This means that $\|f(t)\|_q = \| \mathring f \|_q$ for every $t\in[0,T]$ and ${1\le q\le\infty}$. Moreover we can choose some fixed $k\in\NN$ such that
$$\|f-f_{B_k}\|_{W^{1,2}(0,T;C_b)} + \|f-f_{B_k}\|_{C([0,T];C^1_b)} \le 1\;.$$
Lemma \ref{ES1} implies that $\|f_{B_k}\|_{C([0,T];C^1_b)} \le C$, $\|f_{B_k}\|_{W^{1,2}(0,T;C_b)} \le C$ and $\|f_{B_k}\|_{L^\infty(0,T;W^{2,\beta})} \le C$. Thus
\begin{align*}
\|f\|_{W^{1,2}(0,T;C_b)} &\le \|f-f_{B_k}\|_{W^{1,2}(0,T;C_b)} + \|f_{B_k}\|_{W^{1,2}(0,T;C_b)} \le C, \\
\|f\|_{C([0,T];C^1_b)} &\le \|f-f_{B_k}\|_{C([0,T];C^1_b)}+ \|f_{B_k}\|_{C([0,T];C^1_b)}\le C.
\end{align*}
Moreover by the weak-* lower semicontinuity of the norm,
\begin{align*}
\|f\|_{L^\infty(0,T;W^{2,\beta})} &\le \underset{k\to\infty}{\lim\inf} \|f_{B_k}\|_{L^\infty(0,T;W^{2,\beta})}
\le C. 
\end{align*}
This proves (a) which includes condition (i). Finally, uniqueness follows directly from Proposition \ref{UNIQS}.
\qed\end{proof}

Now that we have showed that any magnetic field $B\in\BB$ induces a unique strong solution of the initial value problem \eqref{VPSU}, we can define an operator mapping every admissible field onto its induced state.

\begin{definition}
\label{CSO}
The operator
\begin{align*}
f. \;\colon \BB \to C([0,T];L^2(\RR^6)),\; B \mapsto f_B
\end{align*}
is called the \textbf{field-state operator}. At this point $f_B$ denotes the unique strong solution of \eqref{VPSU} that is induced by the field $B\in\BB$.
\end{definition}

From now on the notation $f_B$ is to be understood as the value of the field-state operator at point $B\in\BB$.

\section[Continuity and compactness of the field-state operator]{Continuity and compactness of the field-state operator}

Obviously the Lipschitz estimates of Lemma \ref{LIP} hold true for the strong solutions by approximation. 

\begin{corollary}
\label{WLIP}
Let $L_1,\,L_2,\,L_3,\,c_2,\,c_4,\,c_6$ be the constants from Lemma \ref{ES1} and Lemma~\ref{LIP}. Then for all $B,H\in\BB$,\vspace{-2mm}
\begin{gather*}
\begin{aligned}
\|f_B-f_H\|_{C([0,T];C_b)} &\;\le\; L_1 \|B-H\|_\WW\;,\\
\|\partial_z f_B-\partial_z f_H\|_{C([0,T];C_b)} &\;\le\; L_2 \|B-H\|_\WW^{\gamma}\;,\\
\|\partial_t f_B-\partial_t f_H\|_{L^2(0,T;C_b)} &\;\le\; L_3 \|B-H\|_\WW^{\gamma}\;,
\end{aligned}\\
\|\delz f_B\|_{C([0,T];C_b)} \le c_2,\;\; \|\delt f_B\|_{L^2(0,T;C_b)} \le c_4,\;\; \|D_z^2 f_B\|_{L^\infty(0,T;W^{2,\beta})} \le c_6.
\end{gather*}
\end{corollary}

The proof of this Corollary is obvious. I states that the field-state operator is globally Lipschitz-continuous with respect to the norm on $C\big([0,T];C_b\big)$ and globally Hölder-continuous with exponent $\gamma=\gamma(\beta)$ with respect to the norm on $W^{1,2}\big(0,T;C_b\big)$ and the norm on $C\big([0,T];C^1_b\big)$. \pskip

The following proposition provides (weak) compactness of the field-state operator that will be very useful in terms of variational calculus.

\begin{proposition}
\label{COMP}
Let $(B_k)_{k\in\NN} \subset \BB$ be a sequence that is converging weakly in $\WW$ to some limit $B\in\BB$. Then it has a subsequence $(B_{k_j})$ of $(B_k)$ such that
\begin{align*}
f_{B_{k_j}} &\rightharpoonup f_B \quad\text{in}\;\; W^{1,2}(0,T;L^p)\cap L^2(0,T;W^{1,p}) \cap L^2(0,T;W^{2,\beta}),\, 1\le p <\infty,\\
f_{B_{k_j}} &\to f_B \quad\text{in}\;\; L^2([0,T]\times\RR^6)
\end{align*}
if $j$ tends to infinity.
\end{proposition}

\begin{proof} Suppose that $(B_k)_{k\in\NN} \subset \BB$ and $B\in\BB$ such that $B_k \rightharpoonup B$ in $\WW$. By Theorem~\ref{GWS}, $f_k:=f_{B_k}$ is bounded in $W^{1,2}(0,T;L^p) \cap L^2(0,T;W^{1,p}\cap W^{2,\beta})$ for every $1\le p \le\infty$. Note that this bound can be chosen independent of $p$. Hence the Banach-Alaoglu theorem and Cantor's diagonal argument imply that $(f_k)$ converges weakly in $W^{1,2}(0,T;L^m) \cap L^2(0,T;W^{1,m}\cap W^{2,\beta})$,\linebreak for every integer $m \ge 2$, up to a subsequence. Thus there exists some function \linebreak $f\in W^{1,2}(0,T;L^m) \cap L^2(0,T;W^{1,m}\cap W^{2,\beta})$ for every integer $m\ge 2$ such that
\begin{align*}
f_k \rightharpoonup f \;\;\text{in}\;\; W^{1,2}(0,T;L^m)\cap L^2(0,T;W^{1,m}) \cap L^2(0,T;W^{2,\beta}),\;\; m\in\NN, m\ge 2.
\end{align*}
Thus, by interpolation, $f\in W^{1,2}(0,T;L^p)\cap L^2(0,T;W^{1,p}) \cap L^2(0,T;W^{2,\beta})$ for $2\le p <\infty$. We will now show that $f$ is a strong solution to the field $B$ by verifying the conditions from Definition \ref{WKS}. \pskip

\textit{Condition }(iv): Let $\eps>0$ be arbitrary. We will now assume that there exists some measurable set $M\subset [0,T]\times\big( \RR^6\setminus \BR \big)$ with Lebesgue-measure $\lambda(M)>0$ such that $f>\eps$ almost everywhere on $M$. Then
\begin{align*}
0 < \eps \lambda(M) < \int\limits_M f \;\mathrm d(t,z) = \int\limits_M f-f_k \;\mathrm d(t,z) = \int (f-f_k) \mathds{1}_{M} \;\mathrm d(t,z) \to 0
\end{align*}
as $k\to\infty$ which is a contradiction. The case $f<-\eps$ can be treated analogously. Hence $-\eps<f<\eps$ almost everywhere on $[0,T]\times\big( \RR^6\setminus \BR \big)$ which immediately yields $f=0$ almost everywhere on $[0,T]\times\big( \RR^6\setminus \BR \big)$ because $\eps$ was arbitrary. Since $W^{1,2}(0,T;L^p)$ is continuously embedded in $C([0,T];L^p)$ by Sobolev's embedding theorem, we have $\supp f(t) \subset \BR$ even for all $t\in[0,T]$. \pskip

\textit{Condition }(i): The fact that $\supp f(t) \subset B_R(0)$ for all $t\in[0,T]$ directly implies that 
\begin{align*}
f_k \rightharpoonup f \quad\text{in}\quad W^{1,2}(0,T;L^p)\cap L^2(0,T;W^{1,p})\cap L^2(0,T;W^{2,\beta}), \; 1\le p < \infty\,.
\end{align*}
The inequality $\|f\|_{W^{1,2}(0,T;L^p(\RR^6))} + \|f\|_{L^2(0,T;W^{1,p})} \le C$
where $C>0$ depends only on $\mathring f$, $T$, $K$ and $\beta$ follows directly from the weak convergence and the weak lower semicontinuity of the norm. Since $C$ does not depend on $p$ this inequality holds true for $p=\infty$.\pskip

\textit{Condition }(iii): It holds that $f_k \rightharpoonup f$ in $W^{1,2}(0,T;L^2)$ with $f_k(0) = \mathring f$ almost everywhere on $\RR^6$ for all $k\in\NN$. By Mazur's lemma we can construct some sequence $(f_k^*)_{k\in\NN}$ such that $f_k^*\to f$ in $W^{1,2}(0,T;L^2)$ where for any $k\in\NN$, $f_k^*$ is a convex combination of $f_1,...,f_k$. Then of course $f_k^*(0)= \mathring f$ almost everywhere on $\RR^6$ as well and hence
\vspace{-0.1cm}
\begin{align*}
\|f(0)-\mathring f\|_{L^2} &= \|f(0)-f_k^*(0)\|_{L^2} \le C\; \|f-f_k^*\|_{W^{1,2}(0,T;L^2)} \to 0, \quad k\to\infty\;.\\[-0.75cm]
\end{align*}
Thus $f(0)=\mathring f$ almost everywhere on $\RR^6$.\pskip

\textit{Condition }(ii): We know that $f_k\rightharpoonup f$ in $W^{1,2}(0,T;L^2)\cap L^2(0,T;W^{1,2})$ that is $ H^1(]0,T[\times\RR^6)$ due to Lemma \ref{SOBBOCH}. Then, because of the compact support, the Rellich-Kondrachov theorem implies that $f_k\to f$ in $L^2([0,T]\times\RR^6)$, up to a subsequence. From Lemma \ref{NPOT}\,(b) we can conclude that for any $t\in[0,T]$,
\begin{align*}
&\|\delx\psi_{f}(t) - \delx\psi_{f_k}(t)\|_{L^2(\BR)} \le C\; \|f(t)-f_k(t)\|_{L^2} \to 0,\quad k\to\infty\;.
\end{align*}
For brevity, we will now use the notation
\begin{align*}
\mathbf V( \varphi,f,B) := \delt  \varphi + v\cdot\delx  \varphi - \delx\psi_f\cdot \delv  \varphi + (v\times B)\cdot\delv \varphi\;.
\end{align*}
Let $ \varphi\in C_c^\infty(]0,T[\times \RR^6))$ be an arbitrary test function. Then $\mathbf V( \varphi,f_k,B_k)$ is bounded in $L^2(]0,T[\times\BR)$ uniformely in $k$ (the bound may depend on $\varphi$). It also holds that
$$\mathbf V( \varphi,f,B_k) - \mathbf V( \varphi,f_k,B_k) \to 0, \quad k\to \infty \quad\text{in}\; L^2(]0,T[\times\BR)$$ 
since $\psi_{f_k}\to\psi_f$ in $L^2\big(]0,T[\times \BR\big)$. Moreover, 
$$ \mathbf V( \varphi,f,B) -\mathbf V( \varphi,f,B_k) \wto 0,\quad k\to \infty \quad\text{in}\; L^2(]0,T[\times\BR).$$
By integration by parts,
\begin{align*}
&\int\limits_0^T\hspace{-5pt}\int \mathbf V(f,f,B)\; \varphi \dz\,\mathrm dt   = \int\limits_0^T\hspace{-5pt}\int f\;\mathbf V( \varphi,f,B) - f_k\;\mathbf V( \varphi,f_k,B_k) \dz\,\mathrm dt 
\end{align*}
and the right-hand side converges to zero as $k\to\infty$. As $\varphi$ was arbitrary this implies that $\mathbf V(f,f,B) = 0$ almost everywhere on $[0,T]\times\RR^6$ that is (ii).\pskip

Consequently $f$ is a strong solution to the field $B$ and thus $f=f_B$ because of uniqueness. Furthermore we have showed that there exists a subsequence $(B_{k_j})$ of $(B_k)$ such that $(f_{B_{k_j}})$ is converging in the demanded fashion. \end{proof}

\section{A general inhomogenous linear Vlasov equation}

Since a Fr\'echet derivative is a linear approximation, we will find out later that the derivative of the field-state operator is determined by an inhomogenous linear Vlasov equation. In this section we will analyze those linear Vlasov equations in general, i.e., we will establish some existence and uniqueness results. The type and the regularity of the solution will depend on the regularity of the coefficients.\pskip 

Let $r_0\ge 0$ and $r_2>r_1 \ge 0$ be arbitrary. We consider the following inhomogenous linear version of the Vlasov equation:
\begin{displaymath}
\label{LVL}
\refstepcounter{equation}
\delt f + v\scdot\delx f + \MA\scdot\delv f + (v\stimes \MB )\scdot\delv f = \delx\psi_f\scdot\MC + \Mchi\,\PHI_{\Ma,f} + \Mb,\;\; 
f\big\vert_{t=0} = \Mf\;\; \textnormal{(\theequation)}
\end{displaymath}
The coefficients are supposed to have the following regularity
\begin{align}
&\label{RCa} \Ma=\Ma(t,x,v) \in C\big([0,T];C^1_b(\RR^6)\big),	\\
&\label{RCb} \Mb=\Mb(t,x,v) \in C\big([0,T];C^1_b(\RR^6)\big),	\\
&\label{RCf} \Mf=\Mf(x,v) \in C^2_c(\RR^6), \\
&\label{RCA} \MA=\MA(t,x)  \in C\big([0,T];C^{1,\gamma}(\RR^3;\RR^3)\big),	\\
&\label{RCB} \MB=\MB(t,x)  \in C\big([0,T];C^{1,\gamma}(\RR^3;\RR^3)\big),	\\
&\label{RCC} \MC=\MC(t,x,v) \in C\big(0,T;C^1_b(\RR^6;\RR^3)\big), \\
&\label{RCchi} \Mchi=\Mchi(x,v) \in C^1_c(\RR^6;[0,1])
\end{align}
with\vspace{-3mm}
\begin{gather}
\label{SC1} \supp \Ma(t),\; \supp \Mb(t),\; \supp\Mf,\; \supp \MC(t) \subset \Bro,\quad t\in[0,T],\\
\label{SC2} \Mchi = 1 \;\;\text{on}\;\; B_{r_1}(0),\quad \supp \Mchi  \subset B_{r_2}(0)
\end{gather}
Moreover $\PHI_{\Ma,f}$ is given by
\begin{align}
\label{DEFPHI}
\PHI_{\Ma,f}(t,x) := -\iint \frac{x-y}{|x-y|^3}\cdot\delv \Ma(t,y,w) \, f(t,y,w) \;\mathrm dw \mathrm dy 
\end{align}
for all $(t,x)\in[0,T]\times\RR^3$. We will also use the notation
\begin{align}
\label{DEFPHI2}
\PHI_{\Ma,f}'(t,x) := - \iint \frac{x-y}{|x-y|^3}\cdot\Big( \delv \Ma \, \delx f- \delv f \, \delx \Ma \Big)(t,y,w) \;\mathrm dw \mathrm dy\;.
\end{align}
for $(t,x)\in[0,T]\times\RR^3$. Note that
\begin{align*}
\PHI_{\Ma,f} = \sum_{i=1}^3 \delxi \psi_{\delvi \Ma f} \tand \big[\PHI_{\Ma,f}'\big]_j = \sum_{i=1}^3 \delxi \psi_{\delvi \Ma\,\delxj f - \delvi f\,\delxj \Ma},\;\; j=1,2,3.
\end{align*}
As $\Ma \in C\big([0,T];C^1_b(\RR^6)\big)$ with compact support $\supp \Ma(t) \subset \Bro$ for all ${t\in[0,T]}$, Lemma \ref{NPOT} provides the following inequalities: For any $r>0$ there exists some constant $c>0$ that may depend only on $r$ and $r_0$ such that for almost all $t\in[0,T]$,
\begin{align}
&\label{ESTPHI1}	\|\PHI_{\Ma,f}(t)\|_{L^2(B_r(0))} \le c \|\delv a(t)\|_\infty\, \|f(t)\|_{L^2(\Bro)}, && f\hspace{-2pt}\in L^2(0,T;L^2),\\
&\label{ESTPHI2}	\|\PHI_{\Ma,f}'(t)\|_{L^2(B_r(0))} \le c\|\delz a(t)\|_\infty\, \|\delz f(t)\|_{L^2(\Bro)}, \hspace{-7pt}&& f\hspace{-2pt}\in L^2(0,T;H^1),\\
&\label{ESTPHI3}	\|\PHI_{\Ma,f}(t)\|_{L^\infty} \le c \|\delv a(t)\|_\infty\, \|f(t)\|_{L^\infty(\Bro)}, && f\hspace{-2pt}\in L^2(0,T;L^\infty),\\
&\label{ESTPHI4}	\|\PHI_{\Ma,f}'(t)\|_{L^\infty} \le c \|\delz a(t)\|_\infty\, \|\delz f(t)\|_{L^\infty(\Bro)}, && f\hspace{-2pt}\in L^2(0,T;W^{1,\infty}).
\end{align}
If $\Ma \in C\big([0,T];C^2_b(\RR^6)\big)$ and $f\in C\big([0,T];C^1_b(\RR^6)\big)$ then $\PHI_{\Ma,f}$ is continuously differentiable with respect to $x$ with
\begin{align*}
\delxj \PHI_{\Ma,f}(t,x) &= \sum_{i=1}^3 \delxj \delxi \psi_{\delvi \Ma f} 
= \sum_{i=1}^3 \delxi \psi_{\delvi \Ma\, \delxj f - \delxj \Ma\, \delvi f} = \big[\PHI_{\Ma,f}'\big]_j(t,x)
\end{align*}
for all $(t,x)\in [0,T]\times\RR^3$. Because of density this result holds true if \linebreak$\Ma \in C\big([0,T];C^1_b(\RR^6)\big)$. If merely $f\in L^2(0,T;H^1)$ the result holds true in the weak sense. 

\smallskip

%
%

\begin{lemma}
	\label{LVLZ}
	\hypertarget{HLVLZ}
	Let $A,B\in C\big([0,T];C^1_b(\RR^3;\RR^3)\big)$ be arbitrary. For any ${t\in[0,T]}$ and $z\in\RR^6$ the characteristic system
	\begin{align*}
	\dot x = v, \quad
	\dot v = \MA(s,x)+v\times \MB(t,x)\;,
	\end{align*}
	has a unique solution $Z\in C^1([0,T]\times[0,T]\times\RR^6;\RR^6)$, $Z(s,t,z)=(X,V)(s,t,z)$ to the initial value condition $Z(t,t,z)=z$. For any $r>0$ and all $s,t\in [0,T]$,
	$$Z(s,t,B_r(0)) \subset B_{\zeta(r)}(0) \twith \zeta(r) :=  \ee^{2T} \big(r + \sqrt{T} \|\MA\|_{L^2(0,T;L^\infty)} \big) \;.$$
	Moreover, there exists some constant $C(r)\hspace{-2pt}>\hspace{-2pt}0$ depending only on $\|\MA\|_{L^2(0,T;C^1_b)}$, $\|\MB\|_{L^2(0,T;C^1_b)}$ and $r$ such that for all $s,t\in[0,T]$,
	\begin{align*}
	\|\delz Z(s,t,\cdot)\|_{L^\infty(B_r(0))} \le C(r) \tand \|\delt Z(s,t,\cdot)\|_{L^\infty(B_r(0))} \le C(r)\;.
	\end{align*}
\end{lemma}

The proof is simple and very similar to the proof of Lemma \ref{ES1}. Therefore it will not be presented.\pskip

Now we can establish an existence and uniqueness result for classical solutions of the system \eqref{LVL} if the regularity conditions \eqref{RCa}-\eqref{RCchi} hold. Unfortunately the coefficients of the systems that will occur in this paper do not satisfy those strong conditions. However, we will still be able to prove an existence and uniqueness result for strong solutions of \eqref{LVL} if the regularity conditions are slightly weaker.

%
%

\begin{proposition}
	\label{CSLVL}
	\hypertarget{HCSLVL}
	Suppose that the coefficients of the system \eqref{LVL} satisfy the regularity conditions \eqref{RCa}-\eqref{RCchi} and the support conditions \eqref{SC1},\eqref{SC2}. Then the initial value problem \eqref{LVL} has a unique classical solution ${f\in C^1([0,T]\times\RR^6)}$. Moreover for all $t\in[0,T]$, $\supp f(t) \subset B_{\zeta(r+1)}(0)$ with $r=\max\{r_0,r_2\}$ and $f$ is implicitely given by\vspace{-2mm}
	\begin{align}
	\label{EXPLVL}
	f(t,z) = \Mf\big( Z(0,t,z) \big) + \int\limits_0^t \big[\delx\psi_f\cdot\MC + \Mchi\PHI_{\Ma,f} + \Mb\big]\big(s,Z(s,t,z) \big)  \ds
	\end{align}
	for any $t\in[0,T],\;z\in\RR^6$.
	Moreover, there exists some constant $C>0$ depending only on $T$, $r_0$, $r_2$ and the standard norms of the coefficients such that $$\|f\|_{C^1_b([0,T]\times\RR^6)} \le C.$$ 
\end{proposition}
\newpage
\begin{remark} 
	$\;$
	\begin{itemize}
		\itema	If we use a final value condition $f\big|_{t=T}=\Mf$ instead of the initial value condition $f\big|_{t=0}=\Mf$ the problem can be treated completely analogously. The results of Proposition~\ref{CSLVL} and Corollary~\ref{WSLVL} hold true in this case. Only the implicit depiction of a classical solution must be replaced by
		\begin{align}
		\label{EXPLVL2}
		f(t,z) = \Mf\big( Z(T,t,z) \big) - \int\limits_t^T \big[\delx\psi_f\cdot\MC + \Mchi\PHI_{\Ma,f} + \Mb\big]\big(s,Z(s,t,z) \big)  \ds
		\end{align}
		\itemb Suppose that $\MC=0$ and recall that $\PHI_{\Ma,f}$ depends only on $f\big|_{\Bro}$. Hence, if we choose $r_1 = \zeta(r_0)$ then for all $t\in [0,T]$ and $z\in \Bro$,
		\begin{align}
		\label{EXPLVL3}
		f(t,z) = \Mf\big( Z(0,t,z) \big) + \int\limits_0^t \big[\PHI_{\Ma,f} + \Mb\big]\big(s,Z(s,t,z) \big)  \ds
		\end{align}
		because in this case $\Mchi\big(Z(s,t,z)\big) = 1$ as $Z(s,t,\Bro) \subset B_{r_1}(0)$. This means that the values of $f\big\vert_\Bro$ do not depend on the choice of $\Mchi$ as long as \eqref{RCchi} and \eqref{SC2} hold.
	\end{itemize}
\end{remark}

\begin{proof}\hspace{-2pt}\textit{of Proposition \ref{CSLVL}}\;
	Let $c>0$ denote a generic constant depending only on $r_0$, $r_2$, $T$ and the norms of the coefficients. For $t\in[0,T]$ and $z\in \RR^6$ let ${Z=(X,V)(s,t,z)}$ denote the solution of the characteristic system with $Z(t,t,z)=z$. Moreover, for $t\in[0,T]$ and $z\in\RR^6$, we define a recursive sequence by $f_0(t,z) := \Mf(z)$ and
	\begin{align*}
	f_{n+1}(t,z) &:= \Mf(Z(0,t,z)) + \int\limits_0^t \big[ \delx\psi_{f_n}\cdot\MC +\Mchi\PHI_{\Ma,f_n} + \Mb \big]\big(s,Z(s,t,z)\big)  \ds.
	\end{align*}
	By induction we can conclude that all $f_n$ are continuous. Then for any fixed $\tau\in[0,T]$ and $n\in\NN$ the functions $\Mf$, $\big[\delx\psi_{f_n}\cdot\MC\big](\tau)$, $\big[\Mchi\PHI_{\Ma,f_n}\big](\tau)$ and $\Mb(\tau)$ are continuous and compactly supported in $B_{r}(0)$ with $r=\max\{r_0,r_2\}$. This directly implies that $f_0(t)$ is compactly supported with $\supp f_0(t)\subset B_{r}(0)$ for all $t\in[0,T]$. Moreover, for any $\tau\in[0,T]$, Lemma \ref{LVLZ} implies that
	\begin{align*}
	\left.
	\begin{aligned}
	\supp \Mf(Z(s,t,\cdot)) &= Z(t,s,\supp \Mf)\\
	\supp \big[\delx\psi_f\cdot\MC](\tau,Z(s,t,\cdot)) &= Z\big(t,s,\supp \delx\psi_{f_n}\cdot\MC(\tau)\big)\\
	\supp \big[\Mchi\PHI_{\Ma,f_n}\big](\tau,Z(s,t,\cdot)) &= Z\big(t,s,\supp\Mchi\PHI_{\Ma,f_n}(\tau)\big)\\
	\supp \Mb\big(\tau,Z(s,t,\cdot)\big) &= Z\big(t,s,\supp \Mb(\tau)\big)
	\end{aligned}
	\right\} \; \subset \; B_{\zeta(r)}(0)\;.
	\end{align*}
	If we choose $\tau=s$ we can inductively deduce that $\supp f_n(t) \subset B_{\zeta(r)}(0)$ for all ${t\in[0,T]}$ and all $n\in\NN$. Finally, by another induction, ${f_n \in C^1(]0,T[\times\RR^6)}$ as the partial derivatives can be described recursively.
	Using Lemma \ref{LVLZ}, \eqref{ESTPHI3}, \eqref{ESTPHI4} and Lemma \ref{NPOT}, we con conclude by a straightforward computation that there exists some constant $c_*>0$ such that for all $t\in[0,T]$,\vspace{-2mm}
	\begin{align*}
	M_{1,0}(t)\le c_* \quad\text{and}\quad M_{n+1,n}(t) \le c_* \int\limits_0^t M_{n,n-1}(s) \ds,\quad n\in\NN \\[-8mm]
	\end{align*}
	where $M_{m,n}(t)$ denotes the expression
	\begin{align*}
	\max\left\{ \|f_m(t)-f_n(t)\|_\infty,\|\delt f_m(t)-\delt f_n(t)\|_\infty,\|\delz f_m(t)- \delz f_n(t)\|_\infty \right\}
	\end{align*}
	for $m,n\in\NN_0$. Thus by induction,
	\begin{align*}
	M_{n+1,n}(t) \le c_*\frac{t^n}{n!} \le c_*\frac{T^n}{n!}, \quad t\in[0,T],n\in\NN
	\end{align*}
	and hence for $m,n\in\NN$ with $n<m$,\vspace{-2mm}
	\begin{align*}
	M_{m,n}(t) \le \sum_{j=n}^{m-1} M_{j+1,j}(t)  \le  \sum_{j=n}^\infty c_*\frac{T^j}{j!} \to 0,\quad n\to\infty\;.
	\end{align*}
	Consequently $(f_n)$ is a Cauchy-sequence in $C_b^1([0,T]\times\RR^6)$ and converges to some function $f\in C_b^1([0,T]\times\RR^6)$ because of completeness. Obviously, as the radius $\zeta(r)$ does not depend on $n$,
	$\supp f(t) \subset  \overline{B_{\zeta(r)}(0)} \subset B_{\zeta(r+1)}(0)$ for all $t\in[0,T]$
	and $f$ satisfies the equation\vspace{-2mm}
	\begin{align}
	\label{IMPRF}
	f(t,z) = \Mf(Z(0,t,z)) + \int\limits_0^t \big[ \delx\psi_f\cdot\MC + \PHI_{\Ma,f} +\Mb\big](s,Z(s,t,z))\ds\;.
	\end{align}
	One can easily show that $f$ is a classical solution of \eqref{LVL} by differentiating both sides of \eqref{IMPRF} with respect to $t$. 
	We will finally prove uniqueness by assuming that there exists another solution $\tilde f$ of the initial value problem and define $d:=f-\tilde f$. Then for any $t\in[0,T]$,\vspace{-2mm}
	\begin{align*}
	\|d(t)\|_{L^2}^2 = 2\int\limits_0^t\int \delx\psi_{d(s)}\cdot\MC(s)\;d(s) + \Mchi\PHI_{\Ma,d}(s)\; d(s)  \dz\mathrm ds \le c\int\limits_0^t \|d(s)\|_{L^2}^2 \ds
	\end{align*}
	and hence $\|d(t)\|_{L^2}=0$ for all $t\in[0,T]$ by Gronwall's lemma. This directly implies that $f=\tilde f$ which means uniqueness.
\end{proof}

%
%

\begin{definition}
	\label{DWSLVL}
	\hypertarget{HDWSLVL}
	We call $f$ a strong solution of the initial value problem \eqref{LVL} iff the following holds:
	\begin{enumerate}
		\item [\textnormal{(i)}] $f\in H^1(]0,T[\times \RR^6) \subset C([0,T];L^2)$.
		\item [\textnormal{(ii)}] $f$ satisfies
		\begin{align*}
		\delt f  + v\cdot\delx f  + \MA\cdot\delv f + (v\times \MB)\cdot\delv f  = \delx\psi_f\cdot\MC + \PHI_{\Ma,f} + \Mb
		\end{align*}
		almost everywhere on $[0,T]\times\RR^6$.
		\item [\textnormal{(iii)}] $f$ satisfies the initial condition $f\big\vert_{t=0} = \Mf$ almost everywhere on $\RR^6$.
		\item [\textnormal{(iv)}] There exists some radius $r>0$ such that $\supp f(t) \subset B_r (0)$, $t\in [0,T]$.
	\end{enumerate}
\end{definition}

%
%

\begin{corollary}
	\label{WSLVL}
	\hypertarget{HWSLVL}
	We define $r:=\max\{r_0,r_2\}$ and let $C>0$ denote some constant depending only on $r_0,\,r_2$ and the norms of the coefficients.
	\begin{itemize}
		\item [\textnormal{(a)}] Suppose that $\MB  \in L^2(0,T;C^{1,\gamma}(\RR^3;\RR^3)\big)$, $\MC \in L^2\big(0,T;H^1\cap C_b(\RR^6;\RR^3)\big)$, \linebreak $\Mb \in L^2\big(0,T;C_b\cap H^1(\RR^6)\big)$ and $\Mf \in C^1_c(\RR^6)$.
		Moreover, we assume that the regularity conditions \eqref{RCa}, \eqref{RCA}, \eqref{RCchi} and the support conditions \eqref{SC1}, \eqref{SC2} hold.
		Then there exists a unique strong solution $f\in L^\infty\cap H^1(]0,T[\times\RR^6)$ of the initial value problem \eqref{LVL} such that
		\begin{align*}
		\|f\|_{L^\infty(]0,T[\times\RR^6)} + \|f\|_{H^1(]0,T[\times\RR^6)} \le C
		\end{align*}
		and $\supp f(t) \subset B_{\zeta(3+r)}(0)$ for almost all $t\in[0,T]$.
		\item [\textnormal{(b)}] Suppose that $\Mb=0$, $\MC=0$ and $\MB  \in L^2(0,T;C^{1,\gamma}(\RR^3;\RR^3)\big)$.
		Moreover, we assume that the regularity conditions \eqref{RCa}, \eqref{RCf}, \eqref{RCA}, \eqref{RCchi} and the support conditions \eqref{SC1}, \eqref{SC2} hold.
		There exists a unique strong solution $f\in W^{1,2}(0,T;C_b) \cap C([0,T];C^1_b)$ of \eqref{LVL} such that
		\begin{align*}
		\|f\|_{L^\infty(]0,T[\times\RR^6)} + \|f\|_{H^1(]0,T[\times\RR^6)} \le C
		\end{align*}
		and $\supp f(t) \subset B_{\zeta(2+r)}(0)$ for almost all $t\in[0,T]$. If $r_1=\zeta(r_0)$, the values of $f\big\vert_{B_{r_0}(0)}$ do not depend on the choice of $\Mchi$ as long as \eqref{RCchi} and \eqref{SC2} hold. 
	\end{itemize}
\end{corollary}


\begin{proof}
	To prove (a) we can choose $(\Mb_k) \subset C([0,T];C^1_b)$, $(\MB_k) \subset C([0,T];C^{1,\gamma})$, $(\MC_k)\subset C([0,T];C^1_b)$ and $(\Mf_k)\subset C^2_c(\RR^6)$ such that
	\begin{align*}
	&\Mb_k \to \Mb \text{ in } L^2\big( 0,T;C_b\cap H^1 \big),	&& \|\Mb_k\|_{L^2(0,T;H^1)} \le 2 \|\Mb\|_{L^2(0,T;H^1)}, \\
	& && \|\Mb_k\|_{L^2(0,T;C_b)} \le 2 \|\Mb\|_{L^2(0,T;C_b)}, \\
	&\Mf_k \to \Mf \text{ in } C^1_b(\RR^6),	&& \|\Mf_k\|_{C^1_b} \le 2\|\Mf\|_{C^1_b}\\
	&\MB_k \to \MB \text{ in } L^2\big( 0,T;C^{1,\gamma} \big),	&& \|\MB_k\|_{L^2(0,T;C^{1,\gamma})} \le 2\|\MB\|_{L^2(0,T;C^{1,\gamma})} \\
	&\MC_k \to \MC \text{ in } L^2\big( 0,T;C_b \cap H^1\big) ,	&& \|\MC_k\|_{L^2(0,T;H^1)} \le 2 \|\MC\|_{L^2(0,T;H^1)}, \\
	& && \|\MC_k\|_{L^2(0,T;C_b)} \le 2 \|\MC\|_{L^2(0,T;C_b)} 
	\end{align*}
	and for all $t\in[0,T]$, $\supp \Mb_k(t)$, $\supp \Mf_k$, and $\supp \MC(t) \subset B_{r_0+1}(0)$. Then, due to Proposition \ref{CSLVL}, for every $k\in\NN$ there exists a unique classical solution $f_k$ of \eqref{LVL} to the coefficients $\Ma$, $\Mb_k$, $\Mf_k$, $\MA$, $\MB_k$, $\MC_k$ and $\Mchi$. Moreover for all $t\in[0,T]$,
	$\supp f_k(t) \subset B_{\varrho}(0) $ with $\varrho := \zeta(2 + \max\{r_0,r_2\}) = \zeta(2+r)$.
	Now let $Z_k$ denote the solution of the characteristic system to $\MA$ and $\MB_k$ satisfying $Z_k(t,t,z)=z$ and let $c>0$ denote some generic constant depending only on $T$, $r_0$, $r_2$ and the norms of the coefficients. From Lemma \ref{LVLZ} we know that for any $r>0$ and all $s,t\in[0,T]$,
	\begin{align}
	\label{ESZK} 
	\|Z_k(s,t,\cdot)\|_{L^\infty(B_{r}(0))}< C(r) \tand \|\delz Z_k(s,t,\cdot)\|_{L^\infty(B_r(0))}< C(r)
	\end{align}
	where $C(r)>0$ depends only on $r$, $\|\MA\|_{L^2(0,T;C^1_b)}$ and $\|\MB\|_{L^2(0,T;C^1_b)}$. Then we can conclude from the implicit description \eqref{EXPLVL} that\vspace{-2mm}
	\begin{align*}
	|f_k(t,z)| &\le \|\Mf_k\|_\infty + \int\limits_0^t \|\delx\psi_{f_k}(s)\|_\infty \, \|\MC_k(s)\|_\infty + \|\PHI_{\Ma,f_k}(s)\|_\infty + \|\Mb_k(s)\|_\infty \ds\\
	&\le c + c\int\limits_0^t \|{f_k}(s)\|_\infty  \ds, \quad (t,z)\in[0,T]\times\RR^6.
	\end{align*}
	which yields $\|f_k(t)\|_{L^\infty} \le c$ by Gronwall's lemma. By differentiating \eqref{EXPLVL} and using \eqref{ESZK} the $z$-derivative can be bounded similarly by\vspace{-2mm}
	\begin{align*}
	&\|\delz f_k(t)\|_{L^2}^2 = \|\delz f_k(t)\|_{L^2(B_\varrho(0))}^2 \le c + c\; \int\limits_0^t \|\delz f_k(s)\|_{L^2}^2  \ds
	\end{align*}
	which implies that $\|\delz f_k(t)\|_{L^2}\le c$ for all $t\in[0,T]$. Finally one can easily show that $\|\delt f_k\|_{L^2(0,T;L^2)} \le c$ by expressing $\delt f_k$ by the Vlasov equation.
	Since all $f_k(t)$ are compactly supported in $B_\varrho(0)$ this yields 
	$$\|f_k\|_{L^\infty(]0,T[\times\RR^6)} + \|f_k\|_{H^1(]0,T[\times\RR^6)} \le c\;.$$
	Then, according to the Banach-Alaoglu theorem, there exists $f \hspace{-2pt}\in\hspace{-2pt} H^1(]0,T[\times\RR^6)$ such that $f_k \rightharpoonup f$ after extraction of a subsequence. Moreover there exists some function ${f^*\in L^\infty(]0,T[\times\RR^6)}$ such that $f_k \overset{*}{\wto} f^*$ up to a subsequence, i.e., a subsequence of $(f_k)$ converges to $f^*$ with respect to the weak-*-topology on $L^1(]0,T[\times\RR^6)^*$. Thus $f=f^*\in L^\infty(]0,T[\times\RR^6)$  $\cap\; H^1(]0,T[\times\RR^6)$. We will now show that $f$ is a strong solution of \eqref{LVL} by verifying the conditions of Definition~\ref{DWSLVL}.\pskip
	
	\textit{Condition} (i) is evident since ${f \in H^1(]0,T[\times\RR^6)}\subset W^{1,2}(0,T;L^2)$ which directly yields ${f\in C([0,T];L^2)}$ by Sobolev's embedding theorem.\pskip
	
	\textit{Condition} (iv) is also obvious because $\supp f_k \subset B_{\varrho}(0)$ for all $k\in\NN$, $t\in[0,T]$. The radius $\varrho$ does not depend on $k$ and satisfies $\varrho <\zeta(3+r)$.\pskip
	
	\textit{Condition} (ii): By Rellich-Kondrachov, $f_k\to f$ in $L^2([0,T]\times\RR^6)$ up to a subsequence. This implies that $\psi_{f_k} \to \psi_f$ and $\PHI_{\Ma,f_k}\to \PHI_{\Ma,f}$ in $L^2([0,T]\times\RR^3)$ and the assertion easily follows.\pskip
	
	\textit{Condition} (iii): Finally, according to Mazur's lemma, there exists some sequence $(\bar f_k)_{k\in\NN} \subset H^1(]0,T[\times\RR^6)$ such that $\bar f_k\to f$ in $H^1(]0,T[\times\RR^6)$ where for all $k\in\NN$, $\bar f_k$ is a convex combination of $f_1,...,f_k$. This means $\bar f_k(0) = \Mf$ and hence
	\begin{align*}
	\|f(0)-\Mf\|_{L^2} \le c\; \|f-\bar f_k\|_{W^{1,2}(0,T;L^2)} \le c\; \|f-\bar f_k\|_{H^1(]0,T[\times\RR^6)} \to 0,\;\; k\to \infty.
	\end{align*}
	
	Consequently $f$ is a strong solution but we still have to prove uniqueness. We assume that there exists another strong solution $\tilde f$ and define $d:=f-\tilde f$. Then, by the fundamental theorem of calculus,\vspace{-2mm}
	\begin{align*}
	\|d(t)\|_{L^2}^2 
	= 2\int\limits_0^t\int \delx\psi_{d(s)}\cdot\MC(s)\;d(s) + \Mchi\PHI_{\Ma,d}(s)\; d(s)  \dz\mathrm ds
	\le c\int\limits_0^t \|d(s)\|_{L^2}^2 \ds
	\end{align*}
	for all $t\in[0,T]$. Hence $\|f(t)-\tilde f(t)\|_{L^2}^2 = \|d(t)\|_{L^2}^2 = 0$ for every $t\in[0,T]$ by Gronwall's lemma. This proves (a).\pskip
	
	To prove (b) we only have to approximate $\MB$. Therefore we choose some sequence $(\MB_k) \subset C([0,T];C^{1,\gamma})$ such that
	\begin{align*}
	\|\MB_k - \MB\|_{L^2(0,T;C^{1,\gamma})}\to 0  \;\;\text{and}\;\; \|\MB_k\|_{L^2(0,T;C^{1,\gamma})} \le 2\|\MB\|_{L^2(0,T;C^{1,\gamma})}, \;\; k\in\NN.
	\end{align*}
	Then for any $k\in\NN$ there exists a unique classical solution $f_k$ of the system \eqref{LVL} to the coefficients $\Ma$, $\Mf$, $\MA$, $\MB_k$ and $\Mchi$ according to Proposition \ref{CSLVL}. Recall that for all $t\in [0,T]$, $\supp f_k(t) \subset B_{\varrho}(0)$ where $\varrho := \zeta(r+1)$ with $r = \max\{r_0,r_2\}$. Again, let $Z_k$ denote the solution of the characteristic system to $\MA$ and $\MB_k$ satisfying $Z_k(t,t,z)=z$ and in the following the letter $c$ denotes some generic positive constant depending only on $T$, $r_0$, $r_2$ and the norms of the coefficients. Now for all $s,t\in[0,T]$ (where $s\le t$ without loss of generality) and $z\in B_{\varrho}(0)$,$\color{white}{\Big|}$\vspace{-3mm}
	\begin{align*}
	|Z_k(s)-Z_j(s)| 
	&\le \int\limits_s^t c\;(1+\|D_x \MA(\tau)\|_\infty+\|D_x \MB_k(\tau)\|_\infty)\; |Z_k(\tau) - Z_j(\tau)| \dtau \\[-1mm]
	&\qquad+ c\int\limits_0^T \|\MB_k(\tau) - \MB_j(\tau)\|_\infty \dtau
	\end{align*}
	which implies that $\|Z_k(s,t,\cdot)-Z_j(s,t,\cdot)\|_{L^\infty(B_\rho(0))} \le c\;\|\MB_k-\MB_j\|_{L^2(0,T;L^\infty)} $. Similarly, for any $i\in\{ 1,...,6 \}$ the difference of the $i$-th derivative can be bounded by
	\begin{align*}
	&|\delzi Z_k(s)- \delzi Z_j(s)| \le c\;  \|\MB_k-\MB_j\|_{L^2(0,T;C^1_b)}^\gamma\\
	&\qquad\qquad\quad + \int\limits_s^t c\;(1+\|\MA(\tau)\|_{C^{1,\gamma}}+\|\MB_k(\tau)\|_{C^{1,\gamma}})\;|\delzi Z_k(\tau)- \delzi Z_j(\tau)| \dtau.
	\end{align*}
	for all $s,t\in[0,T]$ and $z\in B_{\varrho}(0)$. Thus
	\begin{align*}
	\|\delz Z_k(s)- \delz Z_j(s)\|_{L^\infty(B_\varrho(0))} &\le  c\; \|\MB_k-\MB_j\|_{L^2(0,T;C^1_b)}^\gamma\,.
	\end{align*}
	Now for all $t\in[0,T]$, $z\in B_{\varrho}(0)$,
	\begin{align*}
	|f_k(t,z)-f_j(t,z)| &\le \|D\Mf\|_\infty |Z_k(0,t,z)-Z_j(0,t,z)| + c \hspace{-2pt} \int\limits_0^t\hspace{-2pt} \|f_k(\tau)-f_j(\tau)\|_\infty \mathrm d\tau \\
	&\le c\;\|\MB_k-\MB_j\|_{L^2(0,T;L^\infty)} + c \int\limits_0^t \|f_k(\tau)-f_j(\tau)\|_\infty \dtau
	\end{align*}
	and Gronwall's lemma yields
	$\|f_k-f_j\|_{L^\infty(0,T;L^\infty)} \le c\|\MB_k-\MB_j\|_{L^2(0,T;L^\infty)} $.
	Similarly, for all $t\in[0,T],z\in\Br$,
	\begin{align*}
	|\delz f_k(t,z)- \delz f_j(t,z)| \le c\|\MB_k-\MB_j\|_{L^2(0,T;C^1_b)}^\gamma + c \hspace{-2pt} \int\limits_0^t\hspace{-2pt} \|\delz f_k(\tau)- \delz f_j(\tau)\|_\infty \mathrm d\tau
	\end{align*}
	and consequently $\|\delz f_k- \delz f_j\|_{L^\infty(0,T;L^\infty)} \le c\;\|\MB_k-\MB_j\|_{L^2(0,T;C^{1,\gamma})}^\gamma$. By expressing $\delt f_k$ and $\delt f_j$ by their corresponding Vlasov equation we can easily verify the estimate $\|\delt f_k- \delt f_j\|_{L^2(0,T;C_b)} \le c\;\|\MB_k-\MB_j\|_{L^2(0,T;C^{1,\gamma})}^\gamma$. \pskip
	
	This means that $(f_k)$ is a Cauchy sequence in $W^{1,2}(0,T;C_b)\cap C([0,T];C^1_b)$ and thus it converges to some function $f\in W^{1,2}(0,T;C_b)\cap C([0,T];C^1_b)$ because of completeness. Note that for all $t\in[0,T]$, $\supp f(t) \subset B_{\zeta(r+2)}$. From the strong convergence one can easily conclude that $f$ satisfies the system \eqref{LVL} almost everywhere and thus $f$ is a strong solution according to Definition \ref{DWSLVL}. \pskip
	
	Moreover, by the definition of convergence, we can find $k\in\NN$ such that 
	$\|f-f_k\|_{W^{1,2}(0,T;C_b)} + \|f-f_k\|_{C(0,T;C^1_b)} \le 1$
	and consequently
	\begin{align*}
	&\|f\|_{W^{1,2}(0,T;C_b)} + \|f\|_{C(0,T;C^1_b)} \le 1 + \|f_k\|_{W^{1,2}(0,T;C_b)} + \|f_k\|_{C(0,T;C^1_b)} \le c
	\end{align*}
	as the sequence $(f_k)$ is bounded in $C^1_b(]0,T[\times\RR^6)$ according to Proposition \ref{CSLVL} and $\MB_k$ is bounded by $\|\MB_k\|_{L^2(0,T;C^{1,\gamma})} \le 2\|\MB\|_{L^2(0,T;C^{1,\gamma})}$. \bpskip
	
	We will now assume that $r_1 = \zeta(r_0)$. As it has already been discussed in the comment to Proposition \ref{CSLVL} the values of $f_k\vert_\Bro$ do not depend on the choice of $\Mchi$ as long as \eqref{RCchi} and \eqref{SC2} hold. As $f_k\vert_\Bro$ converges to $f\vert_\Bro$ uniformely on $[0,T]\times\Bro$ this result holds true for $f\vert_\Bro$.
\end{proof}

\section{Fr\'echet differentiability of the field-state operator}

Again, let $K>0$ be arbitrary. We can now use the results of Section 5 to establish Fr\'echet differentiability of the control state operator on $\IBB$ (that is the interior of $\BB$).

\begin{theorem}
	\label{FDCSO}
	Let $f.$ be the field-state operator as defined in Definition \ref{CSO}. For all \linebreak$B\in\BB$, $H\in \VV$ there exists a unique strong solution $f_B^H\in L^\infty\cap H^1(]0,T[\times\RR)$ $\subset C([0,T];L^2)$ of the initial value problem
	\begin{displaymath}
	\refstepcounter{equation}
	\label{FDEQ}
	\left\{
	\begin{aligned}
	&\delt f + v\scdot\delx f - \delx\psi_{f_B}\scdot\delv f - \delx\psi_f\scdot\delv f_B
	+ (v\stimes B)\scdot\delv f + (v\stimes H)\scdot\delv f_B = 0 \\[0.25cm]
	&f\big|_{t=0}=0 \hspace{267pt}\textnormal{(\theequation)}
	\end{aligned}
	\right.	
	\end{displaymath}
	with $\supp f(t) \subset B_\varrho(0)$ for all $t\in [0,T]$ and some radius $\varrho>0$ depending only on $T,K,\mathring f$ and $\beta$. Then the following holds:
	\begin{itemize}
		\item [\textnormal{(a)}] The field-state operator $f.$ is Fr\'echet differentiable on $\IBB$ with respect to the norm on $C([0,T];L^2(\RR^6))$, i.e., for any $B\in\IBB$ there exists a unique linear and bounded operator $f'_B: \VV\to C([0,T];L^2(\RR^6))$ such that
		\begin{gather*}
		\forall \eps>0\; \exists \delta > 0\; \forall H\in \VV \text{ with } \|H\|_{\VV} < \delta :\\[2mm]
		B+H \in \IBB\tand \frac{\| f_{B+H} - f_B - f'_B[H] \|_{C([0,T];L^2)}}{\|H\|_{\VV}} < \eps\,.
		\end{gather*}
		The Fr\'echet derivative is given by $f'_B[H]=f_B^H$ for all $H\in\VV$.
		\item [\textnormal{(b)}] For all $B,H\in\IBB$, the solution $f_B^H$ depends Hölder-continuously on $B$ in such a way that there exists some constant $C>0$ depending only on $\mathring f, T, K$ and $\beta$ such that for all $A,B\in\IBB$,
		\begin{align}
		\label{CFF}
		\underset{\|H\|_\VV \le 1}{\sup}\;\|f_A'[H] - f_B'[H]\|_{L^2(0,T;L^2)} \le C \;\|A-B\|_{\VV}^{\gamma}.
		\end{align}
	\end{itemize}
\end{theorem}

\begin{remark}
	As $K>0$ was arbitrary the obove results hold true on $\IBBB$ instead of $\IBB$. Hence they are especially true for $B\in\BB$. 
\end{remark}

\begin{proof} Let $C$ denote some generic positive constant depending only on $\mathring f$, $K$, $T$ and $\beta$. First note that the system \eqref{FDEQ} is of the type \eqref{LVL} whereby the coefficients of \eqref{FDEQ} satisfy the regularity and support conditions of Corollary~\ref{WSLVL}(a). Hence \eqref{FDEQ} has a strong solution $f_B^H\in L^\infty\cap H^1(]0,T[\times\RR^6)$. To prove Fr\'echet differentiability of the field-state operator we must consider the difference $f_{B+H}-f_B$ with $B\in\IBB$ and $H\in\VV$ such that $B+H\in\IBB$. Therefore we will assume that
	$\|H\|_\VV<\delta$ for some sufficiently small $\delta>0$.
	Now we expand the nonlinear terms in the Vlasov equation \eqref{VP} to pick out the linear parts. We have
	\begin{align*}
	&\delx\psi_{f_{B+H}}\cdot\delv f_{B+H} - \delx\psi_{f_{B}}\cdot\delv f_{B}\\
	&\quad = \delx\psi_{f_B}\cdot \delv(f_{B+H}-f_B) + \delx\psi_{(f_{B+H}-f_B)}\cdot\delv f_B + \mathcal R_1,\\[2mm]
	&\big(v\times (B+H)\big)\cdot\delv f_{B+H} - (v\times B)\cdot\delv f_{B} \\
	&\quad = (v\times B)\cdot\delv(f_{B+H}-f_B) + (v\times H)\cdot\delv f_B + \mathcal R_2
	\end{align*}
	where $\mathcal R_1 := \delx\psi_{(f_{B+H}-f_B)}\cdot\delv(f_{B+H}-f_B)$ and $\mathcal R_2 := (v\times H)\cdot\delv(f_{B+H}-f_B)$ are nonlinear remainders. Then $\mathcal R:=\mathcal R_1-\mathcal R_2$ lies in $L^2(0,T;H^1\cap C_b)$ and from Lemma \ref{NPOT} and Corollary \ref{WLIP} we can conclude that $\|\mathcal R\|_{L^2(0,T;L^2)} \le C \|H\|_{\VV}^{1 + \gamma}$. Obviously $f_{B+H}-f_B$ solves the initial value problem
	\begin{displaymath}	
	\refstepcounter{equation}
	\label{EQFR}
	\hspace{-2pt}\left\{
	\begin{aligned}
	&\hspace{-2pt}\delt f + v\scdot\delx f - \delx\psi_{f_B}\scdot\delv f - \delx\psi_f\scdot\delv f_B
	+ (v\stimes B)\scdot\delv f + (v\stimes H)\scdot\delv f_B \hspace{-1pt}= \mathcal R\\[0.25cm]
	&\hspace{-2pt}f\big|_{t=0}=0 \hspace{271pt}\textnormal{(\theequation)}
	\end{aligned}
	\hspace{-2pt}\right.
	\end{displaymath}
	almost everywhere on $[0,T]\times\RR^6$. From Corollary \ref{WSLVL}\,(a) we know that this solution is unique. Also according to Corollary \ref{WSLVL}\,(a) the system
	\begin{align}
	\label{EQFR2}
	\begin{cases}
	\delt f + v\cdot\delx f - \delx\psi_{f_B}\cdot\delv f - \delx\psi_f\cdot\delv f_B + (v\times B)\cdot\delv f = \mathcal R\;,\\[0.25cm]
	f\big|_{t=0}=0\;.
	\end{cases}
	\end{align}
	has a unique strong solution $f_{\mathcal R}$. Then $f_B^H+f_{\mathcal R}$ is a solution of \eqref{EQFR} due to linearity and thus $f_{B+H}-f_B = f_B^H + f_{\mathcal R}$ because of uniqueness. One can easily show that
	\begin{align*}
	\|f_{\mathcal R}(t)\|_{L^2}^2 \le C \int\limits_0^t   \|f_{\mathcal R}(s)\|_{L^2}^2 + \|f_{\mathcal R}(s)\|_{L^2}\;\|\mathcal R(s)\|_{L^2} \ds\;.
	\end{align*}
	Applying first the standard version and then the quadratic version of Gronwall's lemma (cf. Dragomir \cite[p.\,4]{dragomir}) yields
	\begin{align*}
	\|f_{\mathcal R}(t)\|_{L^2} \le C\;\|\mathcal R\|_{L^2(0,T;L^2)} \le C\;\|H\|_{\VV}^{1+\gamma}
	\end{align*}
	Let now $\eps>0$ be arbitrary. Then for all $t\in[0,T]$,
	\begin{align*}
	\frac{\| f_{B+H} - f_B - f_B^H \|_{C([0,T];L^2)}}{\|H\|_{\VV}} = \frac{\| f_{\mathcal R} \|_{C([0,T];L^2)}}{\|H\|_{\VV}}
	\le C\;\|H\|_{\VV}^{\gamma} < \eps
	\end{align*}
	if $\delta$ is sufficiently small. Hence assertion (a) is proved and the Fr\'echet derivative is determined by the system \eqref{FDEQ}. \pskip
	
	To prove (b) suppose that $A,B\in\IBB$ and $H\in\VV$ with $\|H\|_\VV\le 1$. Now, we choose sequences $(A_k),(B_k),(H_k)\subset C([0,T];W^{2,\beta})\subset C([0,T];C^{1,\gamma})$ such that $A_k\to A$, $B_k\to B$, $H_k\to H$ in $L^2(0,T;C^{1,\gamma})$ if $k$ tends to infinity. From Corollary \ref{WSLVL} (and its proof) we can conclude that
	\begin{gather*}
	\|f_{A_k}^{H_k}\|_{H^1(]0,T[\times\RR^6)} \le C \tand \|f_{B_k}^{H_k}\|_{H^1(]0,T[\times\RR^6)}  \le C,\\
	f_{A_k}^{H_k} \rightharpoonup f_A^H \tand f_{B_k}^{H_k} \rightharpoonup f_B^H \quad\text{in}\; H^1(]0,T[\times\RR^6)\,.
	\end{gather*}
	Since the $(x,v)$-supports of all occurring functions are contained in some ball $B_\varrho(0)$ whose radius $r$ depends only on $\mathring f$, $K$, $T$ and $\beta$ but not on $k$, we can apply the Rellich-Kondrachov theorem to obtain\vspace{-0.2cm}
	\begin{align*}
	f_{A_k}^{H_k} \to f_A^H \tand f_{B_k}^{H_k} \to f_B^H \quad\text{in}\; L^2([0,T]\times\RR^6)
	\end{align*}
	up to a subsequence. As $A_k,B_k$ and $H_k$ satisfy the regularity condition \eqref{RCB}, $f_{A_k}^{H_k}$ and $f_{B_k}^{H_k}$ are classical solutions and can be described implicitely by the representation formula \eqref{EXPLVL}. Note that Lemma \ref{ES1} holds true for $\varrho$ instead of $R$. Hence for all $s,t\in[0,T]$,
	\begin{gather*}
	\|Z_{F}(s,t,\cdot)\|_{L^\infty(B_\varrho(0))} \le C, \quad \|\delz f_F(s)\|_\infty \le C, \quad \|D_z^2 f_F\|_{L^2(0,T;L^2)} \le C 
	\end{gather*}
	for all $F\in\big\{A_k,B_k \,\big\vert\, k\in\NN \big\}$. Also recall that we know from Lemma \ref{LIP} (with $\varrho$ instead of $R$) that for all $s,t\in[0,T]$,
	\begin{gather*}
	\|f_{A_k}(s) - f_{B_k}(s)\|_{\infty} \le C\; \|A_k - B_k\|_{\VV}\;,\\[0.5mm]
	\|\delz f_{A_k}(s) - \delz f_{B_k}(s)\|_{\infty} \le C\; \|A_k - B_k\|_{\VV}^{\gamma}\;,\\
	\|Z_{A_k}(s,t,\cdot) - Z_{B_k}(s,t,\cdot)\|_{L^\infty(B_\varrho(0))} \le C\; \|A_k - B_k\|_{\VV}\;.
	\end{gather*}
	Using the implicit description \eqref{EXPLVL} it follows by a simple computation and the application of Gronwall's lemma that
	\begin{align*}
	\|f_{A_k}^{H_k} - f_{B_k}^{H_k}\|_{L^2(0,T;L^2)} \le C\; \|f_{A_k}^{H_k} - f_{B_k}^{H_k}\|_{L^\infty(0,T;L^2)} \le C\; \|A_k - B_k\|_{\VV}^{\gamma}.
	\end{align*}
	If $k\to\infty$, we obtain $\|f_{A}^{H} - f_{B}^{H}\|_{L^2(0,T;L^2)} \le C\; \|A - B\|_{\VV}^{\gamma}$ that is (b). \end{proof}


\section{An optimal control problem with a tracking type cost functional}

Let $\mathring f \in C^2_c(\RR^6)$ be any given initial datum and let $T>0$ denote some fixed final time. The aim is to control the time evolution of the distribution function in such a way that its value at time $T$ matches a desired distribution function $f_d\in C^2_c(\RR^6)$ as closely as possible. More precisely we want to find a magnetic field $B$ such that the $L^2$-difference $\|f_B(T)-f_d\|_{L^2}$ becomes as small as possible. Therefore we intend to minimize the quadratic cost functional:
\begin{align}
\label{OP1}
\begin{aligned}
&\text{Minimize} \quad J(B) = \frac 1 2 \|f_B(T)-f_d\|_{L^2(\RR^6)}^2 + \frac \lambda 2  \|D_x B\|_{L^2([0,T]\times\RR^3;\RR^{3\times 3})}^2, \\&\text{s.t.}\; B\in\BB.
\end{aligned}
\end{align}
where $\lambda$ is a nonnegative parameter. The field $B$ is the control in this model. As the state $f_B(t)$ preserves the $p$-norm, i.e., $\|f_B(t)\|_{p} = \|\mathring f\|_{p}$ for all $1\le p\le \infty$, $t\in[0,T]$, it makes sense to assume that $\|f_d\|_{p} = \|\mathring f\|_{p}$ for all $1\le p\le \infty$ because otherwise the exact matching $f(T)=f_d$ would be foredoomed to fail. \pskip

At first appearance the term ${\frac \lambda 2 \|D_x B\|_{L^2}^2}$ seems to be useless or even counterproductive as we actually want to minimize the expression ${\|f(T)-f_d\|_{L^2}}$. However, in optimal control theory, such a term is usually added because of its smoothing effect on the control. If $\lambda>0$ a magnetic field is punished by high values of the cost functional if its derivatives become large. Of course the weight of punishment depends on the size of $\lambda$. For that reason the additional term is referred to as the regularization term. Note that the regularity $B\in\HH$ is now necessary to avoid infinite values of the cost functional.\pskip

%

Of course such an optimization problem does only make sense if there actually exists at least one globally optimal solution. This fact will be established in the next Theorem. The proof is quite short as most of the work has already been done in the previous sections.
\begin{theorem}
	\label{EXOC1}
	The optimization problem \eqref{OP1} possesses a (globally) optimal solution $\B$, i.e., for all $B\in\BB$, $J(\B)\le J(B)$. 
\end{theorem}

\begin{proof} Suppose that $\lambda>0$ (if $\lambda = 0$ the proof is similar but even easier). The cost functional $J$ is bounded from below since $J(B)\ge 0$ for all $B\in\BB$. Hence ${M:={\inf}_{B\in\BB} J(B)}$ exists and we can choose a minimizing sequence $(B_k)_{k\in\NN}$ such that $J(B_k) \to M$ if ${k\to\infty}$. Without loss of generality we can assume that $J(B_k)\le M+1$ for all $k\in\NN$. As $\BB\subset\VV$ is weakly compact according to Lemma~\ref{LAC} it holds that $B_k\rightharpoonup \B$ in $L^2(0,T;W^{2,\beta})\cap L^2(0,T;H^1)$ for some weak limit $\B\in\BB$ after extraction of a subsequence. Then we know from Proposition~\ref{COMP} that $f_{B_k} \wto f_\B $ in $W^{1,2}(0,T;L^2)$ after subsequence extraction. By the fundamental theorem of calculus this implies that $f_{B_k}(T)\wto f_{\B}(T)$ in $L^2(\RR^6)$. Together with the weak lower semicontinuity of the $L^2$-norm this yields
	\begin{align*}
	J(\B) & \le \underset{k\to\infty}{\lim\inf}\left[ \frac 1 2 \|f_{B_k}(T)-f_d\|_{L^2}^2 \right] + \underset{k\to\infty}{\lim\inf} \left[ \frac \lambda 2  \|D_x B_k\|_{L^2}^2 \right]\\
	& \le \underset{k\to\infty}{\liminf}\left[ \frac 1 2 \|f_{B_k}(T)-f_d\|_{L^2}^2 +  \frac \lambda 2 \|D_x B_k\|_{L^2}^2 \right]
	= \underset{k\to\infty}{\lim}\; J(B_k) = M.
	\end{align*}
	By the definition of infimum this proves $J(\B) = M$. \end{proof}

Of course this theorem does not provide uniqueness of a globally optimal solution. In general, the optimization problem may have more than one globally optimal solution and, of course, it may have more than one locally optimal solution. Therefore we will characterize the local minimizers in the following subsections by necessary and sufficient conditions.

\subsection{Necessary conditions for local optimality}

A locally optimal solution is defined as follows:

\begin{definition}
	A control $\B\in\BB$ is called a locally optimal solution of the optimization problem \eqref{OP1} iff there exists $\delta>0$ such that
	\begin{align*}
	J(\B)\le J(B) \quad\text{for all}\quad B\in B_\delta(\B) \cap \BB
	\end{align*}
	where $B_\delta(\B)$ is the open ball in $\VV$ with radius $\delta$ and center~$\B$.\pskip
\end{definition}

To establish necessary optimality conditions of first order we need Fréchet differentiability of the cost functional $J$.

\begin{lemma}
	\label{FDJ-1}
	The cost functional $J$ is Fréchet differentiable on $\BB$ with Fréchet derivative
	\begin{align*}
	J'(B)[H] = \langle f_B(T)-f_d , f_B'(T)[H] \rangle_{L^2(\RR^6)} + \lambda \langle D_x B, D_x H \rangle_{L^2([0,T]\times\RR^3;\RR^{3\times 3})}
	\end{align*}
	for all $H\in \VV$. Let $\B\in\BB$ be a locally optimal solution of the optimization problem \eqref{OP1}. Then
	\begin{align*}
	J'(\B)[H] \begin{cases} =0, &\text{if}\; \B\in\IBB \\ \ge 0, &\text{if}\; \B\in\partial\BB \end{cases}, \quad H\in\BB \text{ with } \B+H \in \BB.
	\end{align*}
\end{lemma}

\begin{proof} As the control-state operator is Fréchet differentiable on $\BB$ so is the cost functional $J$ by chain rule. Thus, the function $[0,1]\ni t\mapsto J(\B+tH)\in\RR$ is differentiable with respect to $t$ and since $\B$ is also a local minimizer of this function, we have
	\begin{align*}
	0 &\le \ddt J(\B + tH)\big\vert_{t=0} = \left(J'(\B+tH)\left[\ddt(\B+tH)\right]\right)\Big\vert_{t=0} = J'(\B)[H]
	\end{align*}
	for any $H\in\BB$ with $B+H\in\BB$. If $\B$ is an inner point of $\BB$ this line even holds with "$=$" instead of "$\le$".
\end{proof}

If we consider $\BB$ as a subset of $L^2([0,T]\times \RR^3;\RR^3)$ it might be possible to find an adjoint operator $\big(f_B'(T)\big)^*\hspace{-3pt}: C([0,T];L^2) \to L^2([0,T]\times \RR^3;\RR^3)$ of $f_B'(T)$. Then, by integration by parts,\vspace{-4mm}
\begin{align*}
J'(B)[H] &= \langle f_B(T)-f_d , f_B'(T)[H] \rangle_{L^2(\RR^6)} + \lambda \sum_{i=1}^3 \langle \delxi B, \delxi H \rangle_{L^2([0,T]\times\RR^3;\RR^3)} \\
& = \langle \big(f_B'(T)\big)^*[f_B(T)-f_d]  - \lambda\;  \laplace_x B,H \rangle_{L^2([0,T]\times\RR^3;\RR^3)}
\end{align*}
for all $H\in\VV$. This means that the derivative $J'$ would have the explicit description $J'(B) = \big(f_B'(T)\big)^*[f_B(T)-f_d]  - \lambda\; \laplace_x B$. If now $\B \in \interior \BB$ were a locally optimal solution it would satisfy the semilinear Poisson equation
\begin{align*}
\label{ADJOP}
-\laplace_x B = -\frac 1 \lambda \big(f_B'(T)\big)^*[f_B(T)-f_d] \;.
\end{align*}
In general such an adjoint operator is not uniquely determined. This means that we cannot deduce uniqueness of our optimal solution. A common technique to find an adjoint operator is the \textbf{Lagrangian technique}. For $B\in\BB$ and $f,g\in H^1(]0,T[\times\RR^6)$ with $\supp f(t)\subset\BR$ for all $t\in[0,T]$ we define
\begin{align*}
\Lag(f,B,g) &:= \frac 1 2 \|f(T)-f_d\|_{L^2} + \frac \lambda 2  \|D_x B\|_{L^2}^2 \\
&\quad - \int\limits_{[0,T]\times \RR^6} \hspace{-10pt}\big( \delt f +v\cdot \delx f - \delx\psi_f\cdot\delv f + (v\times B)\cdot\delv f \big)\, g \dtxv.
\end{align*}
$\Lag$ is called the \textbf{Lagrangian}. Obviously, by integration by parts,
\begin{align*}
\Lag(f,B,g)
&= \frac 1 2 \|f(T)-f_d\|_{L^2} + \frac \lambda 2 \|D_x B\|_{L^2}^2 + \langle g(0),f(0) \rangle_{L^2} - \langle g(T),f(T) \rangle_{L^2}\\
&\quad + \int\limits_{[0,T]\times \RR^6}\hspace{-10pt} \big( \delt g +v\cdot \delx g - \delx\psi_f\cdot\delv g + (v\times B)\cdot\delv g \big)\; f \dtxv.
\end{align*}
In the definition of the Lagrangian $f$, $B$ and $g$ are independent functions. However, inserting $f=f_B$ yields
\begin{align}
J(B) = \Lag(f_B,B,g), \quad B\in\BB,\; g\in H^1(]0,T[\times\RR^6)\;.
\end{align}
It is important that this equality does not depend on the choice of $g$. Since $\Lag$ is Fréchet differentiable with respect to $f$ in the $H^1(]0,T[\times\RR^6)$-sense and with respect to $B$ in the $\VV$-sense we can use this fact to compute the derivative of $J$ alternatively. By chain rule,
\begin{align}
\label{DJLAG}
J'(B)[H] = \big(\partial_f \Lag\big)(f_B,B,g)\big[f_B'[H]\big] + \big(\partial_B \Lag\big)(f_B,B,g)[H]
\end{align}
for all $B\in\BB$, $H\in\VV$ and any $g\in H^1(]0,T[\times\RR^6)$. Here $\partial_f \Lag$ and $\partial_B \Lag$ denote the partial Fréchet derivative of $\Lag$ with respect to $f$ and $B$. We will now fix $f,g$ and $B$. Then
\begin{align}
\label{DFLAG}
(\partial_f \Lag)(f,B,g)[h]
%
%
&= \langle f(T)-f_d, h(T) \rangle_{L^2} - \langle g(T) , h(T) \rangle_{L^2} + \langle g(0), h(0) \rangle_{L^2} \notag\\
&\; + \hspace{-4pt} \int\limits_{[0,T]\times \RR^6} \hspace{-12pt} \big( \delt g +v\cdot \delx g - \delx\psi_f\cdot\delv g + (v\times B)\cdot\delv g \big)\, h\, \mathrm d(t,x,v) \notag\\
&\; - \hspace{-4pt} \int\limits_{[0,T]\times \RR^6} \PHI_{f,g}(t,x,v)\; h \dtxv
\end{align}
for all $h\in H^1(]0,T[\times\RR^6)$ with $\supp h(t) \subset\BR,\, t\in[0,T]$ where $\PHI_{f,g}$ is given by \eqref{DEFPHI}. Moreover,
\begin{displaymath}
\refstepcounter{equation}
\label{DBLAG}
\begin{aligned}
&(\partial_B \Lag)(f,B,g)[H] = \lambda \langle D_x B, D_x H \rangle_{L^2} \;- \hspace{-4pt}\int\limits_{[0,T]\times \RR^6} (v\times H)\cdot\delv f\; g \dtxv \\
& \quad = \hspace{-4pt} \int\limits_{[0,T]\times \RR^3} \left[ - \lambda \laplace_x B +  \int\limits_{\RR^3} v\times \delv f\; g \dv \right] \cdot H \dtx
\end{aligned} \hspace{9pt}\textnormal{(\theequation)}
\end{displaymath}
for all $H\in\VV$. Apparently, the derivative with respect to $B$ looks pretty nice while the derivative with respect to $f$ is rather complicated. However if we insert those terms in \eqref{DJLAG} we can still choose $g$. Now the idea of the Lagrangian technique is to choose $g$ in such a way that the term $(\partial_f \Lag)(f_B,B,g)[f_B'[H]]$ vanishes. \pskip

We consider the following final value problem which is referred to as the \textbf{costate equation} or the \textbf{adjoint equation}:
\begin{align}\vspace{-2mm}
\label{COSTEQ}
\begin{cases}
\delt g + v\cdot\delx g  - \delx\psi_{f_B}\cdot\delv g + (v\times B)\cdot\delv g = \PHI_{f_B,g}\,\chi \\[0.15cm]
g\big\vert_{t=T}=f_B(T)-f_d
\end{cases}
\end{align}
where $\chi \in C^2_c(\RR^6;[0,1])$ with $\chi = 1$ on $B_{R_Z}(0)$ and $\supp \chi \in B_{2R_Z}(0)$ denotes an arbitrary but fixed cut-off function. Here $R_Z$ is the constant from Lemma \ref{ES1}, i.e., for all $s,t\in[0,T]$, $Z_B(s,t,\BR) \subset \BRZ$. Existence and uniqueness of a strong solution to this system will be established in the following theorem:

%
%

\begin{theorem}
	\label{ADJS}
	\hypertarget{HADJS}
	Let $B\in\BB$ be arbitrary. The costate equation \eqref{COSTEQ} has a unique strong solution
	$g_B\in W^{1,2}\big(0,T;C_b(\RR^6)\big)\cap C\big([0,T];C^1_b(\RR^6)\big) \cap L^\infty\big(0,T;H^2(\RR^6)\big)$ with compact support $\supp g_B(t) \subset B_{R^*}(0)$ for all ${t\in[0,T]}$ and some radius $R^*>0$ depending only on $\mathring f,f_d,T,K$ and $\beta$. \pskip
	
	In this case $g_B\big\vert_{\BR}$ does not depend on the choice of $\chi$. \pskip
	
	Moreover $g_B$ depends Lipschitz/Hölder-continuously on $B$ in such a way that there exists some constant $C\ge 0$ depending only on $\mathring f,f_d,T,K,\beta$ and $\|\chi\|_{C^1_b}$ such that for all $B,H\in\BB$,
	\begin{align}
	\label{HCG}
	\|g_B - g_H\|_{C([0,T];C_b)} &\le C\|B-H\|_{\VV},\\
	\|g_B - g_H\|_{W^{1,2}(0,T;C_b)} + \|g_B - g_H\|_{C([0,T];C^1_b)} &\le C\|B-H\|_{\VV}^{\gamma}.
	\end{align}
\end{theorem}

\begin{remark}
	Note that only the values of $g_B$ on the ball $\BR$ will matter in the following approach. Therefore it is essential that those values are not influenced by the cut-off function $\Mchi$.
\end{remark}

\begin{proof}
	$\,$\textit{Step 1}: Obviously the system \eqref{COSTEQ} has a unique strong solution $g_B$ in the sense of Corollary~\ref{WSLVL}\,(a). Unfortunately the coefficients do not satisfy the stronger regularity conditions of Corollary~\ref{WSLVL}\,(b) as the final value $f_B(T)-f_d$ is not in $C^2_c(\RR^6)$. However, because of linearity, it holds that $g_B = \tg_B - h_B$ where $\tg_B$ is a solution of
	\begin{align*}
	\delt \tg + v\cdot\delx \tg  - \delx\psi_{f_B}\cdot\delv \tg + (v\times B)\cdot\delv \tg = \PHI_{f_B,\tg}\ \chi, \qquad 
	\tg\big\vert_{t=T}=f(T)
	\end{align*}
	and $h_B$ is a solution of
	\begin{align*}
	\delt h + v\cdot\delx h  - \delx\psi_{f_B}\cdot\delv h + (v\times B)\cdot\delv h = \PHI_{f_B,h}\ \chi, \qquad 
	h\big\vert_{t=T}=f_d
	\end{align*}
	Now the first system has a unique strong solution in the sense of Corollary~\ref{WSLVL}\,(a) and the second one possesses a strong solution in the sense of Corollary~\ref{WSLVL}\,(b) since ${f_d\in C^2_c(\RR^6)}$. Indeed the solution $\tg_B$ is much more regular. As $\PHI_{f_B,f_B}=0$ one can easily see that $f_B$ is a solution of the first system and thus, because of uniqueness, $\tg_B=f_B$. Consequently $g_B = f_B - h_B$ lies in the space $W^{1,2}(0,T;C_b)\cap C([0,T];C^1_b)$. Due to Corollary~\ref{WSLVL}\,(b) the values of $h_B$ on $B_R(0)$ do not depend on the choice of $\chi$. Of course $f_B$ does not depend on $\chi$ either and hence $g_B\big\vert_{\BR}$ does not depend on the choice of $\chi$. \pskip
	
	\textit{Step 2}: We will now prove the Hölder estimate. It suffices to establish the result for $h.$ as the result has already been proved for $f.$ in Corollary \ref{WLIP}. Therefore let $B,H\in\BB$ be arbitrary and let $C>0$ denote some generic constant depending only on $\mathring f$, $f_d$, $T$, $K$, $\beta$ and $\|\chi\|_{C^2_b}$. According to Lemma~\ref{SOBBOCH} there exist sequences $(B_k),(H_k)\subset \MM$ such that 
	\begin{align*}
	\|B_k - B\|_{\VV} \to 0, \quad \|H_k - H\|_{\VV} \to 0
	\end{align*}
	if $k\to\infty$. By Corollary \ref{WSLVL}\,(b) (and its proof) the induced strong solutions $h_{B_k}$ and $h_{H_k}$ satisfy\vspace{-2mm}
	\begin{gather*}
	h_{B_k}\to h_B ,\; h_{H_k}\to h_H \quad\text{in}\; W^{1,2}(0,T;C_b) \cap C([0,T];C^1_b),\\
	\|h_{B_k}\|_{W^{1,2}(0,T;C_b)} + \|h_{B_k}\|_{C([0,T];C^1_b)} \le C,\\ \|h_{H_k}\|_{W^{1,2}(0,T;C_b)}+ \|h_{H_k}\|_{C([0,T];C^1_b)} \le  C.
	\end{gather*}
	The constant $C$ does not depend on $k$ since $\|B_k\|_{\VV}$ and $\|H_k\|_{\VV}$ are bounded by $2K$. Also note that there exists some constant $\varrho>0$ depending only on $\mathring f,f_d,T,K$ and $\beta$ (but not on $k$) such that ${\supp h_{B_k} \subset B_{\varrho}(0)}$ and also $\supp h_{H_k} \subset B_{\varrho}(0)$. As $h_{B_k}$ and $h_{H_k}$ are classical solutions they satisfy the implicit representation formula \eqref{EXPLVL2}. We also know from Lemma~\ref{LIP} (with $\varrho$ instead of $R$) that
	\begin{gather*}
	\|f_{B_k}(t) - f_{H_k}(t)\|_\infty \le C\; \|B_k - H_k\|_{\VV}\;,\\
	\|D_z f_{B_k}(t) - D_z f_{H_k}(t)\|_\infty \le C\; \|B_k - H_k\|_{\VV}^\gamma\;,\\
	\|Z_{B_k}(t) - Z_{H_k}(t)\|_{L^\infty(B_\varrho(0))} \le C\; \|B_k - H_k\|_{\VV}\\
	\|D_z Z_{B_k}(t) - D_z Z_{H_k}(t)\|_{L^\infty(B_\varrho(0))}\le C\; \|B_k - H_k\|_{\VV}^\gamma
	\end{gather*}
	for all $t\in[0,T]$. Together with Lemma \ref{ES1} we can show that
	\begin{align*}
	&\|h_{B_k}(t) - h_{H_k}(t)\|_{L^\infty} \\[2mm]
	&\le \|f_d\|_{C^1_b} \;\|Z_{B_k}(T,t,\cdot)-Z_{H_k}(T,t,\cdot)\|_{L^\infty(B_{\varrho}(0))}\\
	&\quad + \int\limits_t^T \|\PHI_{f_{B_k},h_{B_k}}(s,Z_{B_k}(s,t,\cdot)) - \PHI_{f_{H_k},h_{H_k}}(s,Z_{H_k}(s,t,\cdot))\|_{L^\infty(B_{\varrho}(0))}\ds \displaybreak\\
	&\le C\; \|B_k-H_k\|_{\VV}\\
	&\quad + C\; \int\limits_t^T \|h_{B_k}(s)\|_{L^\infty}\; \|D_z f_{B_k}(s)\|_{L\infty}\; \|Z_{B_k}(s)-Z_{H_k}(s)\|_{L^\infty(B_{\varrho}(0))} \ds\\
	&\quad + C\; \int\limits_t^T \|\delv h_{B_k}(s)\|_{L^\infty}\; \|f_{B_k}(s)- f_{H_k}(s)\|_{L^\infty}\; \ds\\
	&\quad + C\; \int\limits_t^T \|h_{B_k}(s)- h_{B_k}(s)\|_{L^\infty}\; \|\delv f_{H_k}(s)\|_{L^\infty}  \ds \\
	&\le C\; \|B_k-H_k\|_{\VV} + C\; \int\limits_t^T \|h_{B_k}(s)- h_{B_k}(s)\|_{L^\infty}\ds
	\end{align*}
	and hence $\|h_{B_k} - h_{H_k}\|_{C([0,T];C_b)} \le C\; \|B_k-H_k\|_{\VV}$. By a similar computation,
	\begin{align*}
	\|\delz h_{B_k}(t) - \delz h_{H_k}(t)\|_{L^\infty} 
	\le C \|B_k-H_k\|_{\VV}^{\gamma} + C\hspace{-1pt}\int\limits_t^T\hspace{-2pt} \|\delz h_{B_k}(s)- \delz h_{B_k}(s)\|_{L^\infty}\;\mathrm ds
	\end{align*}
	and consequently $\|\delz h_{B_k}(t) - \delz h_{H_k}\|_{C([0,T];C_b)} \le C\; \|B_k-H_k\|_{\VV}^{\gamma}$ by Gronwall's lemma.\\ 
	Expressing $\delt h_{B_k}$ and $\delt h_{H_k}$ by their corresponding Vlasov equation then yields \linebreak$\|\delt h_{B_k} - \delt h_{H_k}\|_{L^2(0,T;C_b)} \le C\; \|B_k-H_k\|_{\VV}^{\gamma}$. In summary, we have established that
	\begin{align*}
	\|h_{B_k} - h_{H_k}\|_{C([0,T];C_b)}\le C\; \|B_k-H_k\|_{\VV},\\
	\|h_{B_k} - h_{H_k}\|_{W^{1,2}(0,T;C_b)} + \|h_{B_k} - h_{H_k}\|_{C([0,T];C^1_b)}\le C\; \|B_k-H_k\|_{\VV}^{\gamma}.
	\end{align*}
	For $k\to\infty$ this directly implies that
	\begin{align*}
	\|h_{B} - h_{H}\|_{C([0,T];C_b)}\le C\; \|B-H\|_{\VV},\\
	\|h_{B} - h_{H}\|_{W^{1,2}(0,T;C_b)} + \|h_{B} - h_{H}\|_{C([0,T];C^1_b)}\le C\; \|B-H\|_{\VV}^{\gamma}.
	\end{align*}
	and hence
	\begin{align*}
	\|g_{B} - g_{H}\|_{C([0,T];C_b)}\le C\; \|B-H\|_{\VV},\\
	\|g_{B} - g_{H}\|_{W^{1,2}(0,T;C_b)} + \|g_{B} - g_{H}\|_{C([0,T];C^1_b)}\le C\; \|B-H\|_{\VV}^{\gamma}.
	\end{align*}

	\textit{Step 3}: We must still prove that $g_B\in L^\infty(0,T;H^2)$. As $f_B\in L^\infty(0,T;H^2)$ has already been established in Theorem \ref{GWS} it suffices to show that $h_B$ is twice weakly differentiable with respect to $z$ and $D^2_z h_B\in L^\infty(0,T;L^2)$. Recall that for any $k\in\NN$, $f_{B_k}\in C([0,T];C^2_b)$ according to Theorem \ref{GCS} and $h_{B_k}\in C([0,T];C^1_b)$ according to Theorem \ref{CSLVL}. Thus for all $i\in\{1,2,3\}$,
	\begin{gather*}
	\delxi\delvi f_{B_k}\, h_{B_k} + \delvi f_{B_k}\, \delxi h_{B_k} \in C([0,T];C_b),\\
	\supp \big[\delxi\delvi f_{B_k}\, h_{B_k} + \delvi f_{B_k}\, \delxi h_{B_k}\big](t) \subset \BR, \quad \text{for all}\; t\in [0,T],\\
	\psi_{(\delxi\delvi f_{B_k}\, h_{B_k} + \delvi f_{B_k}\, \delxi h_{B_k})} \in C\big([0,T];C^1_b\big) \cap C\big([0,T];H^2(B_{r}(0))\big)\,,\quad r>0.
	\end{gather*}
	The third line follows from Lemma \ref{NPOT}. Consequently,
	\begin{align*}
	&\PHI_{f_{B_k},h_{B_k}} = \sum_{i=1}^3 \delxi\psi_{\delvi f_{B_k}\,h_{B_k}} = \sum_{i=1}^3 \psi_{(\delxi\delvi f_{B_k}\, h_{B_k} + \delvi f_{B_k}\, \delxi h_{B_k})}
	\end{align*}
	lies in $C\big([0,T];C^1_b\big) \cap C\big([0,T];H^2(B_{r}(0))\big)$ for any $r>0$ and hence
	\begin{align}
	\label{EQZBK}
	\PHI_{f_{B_k},h_{B_k}}\,\chi \in C\big([0,T];C^1_b\big) \cap C\big([0,T];H^2\big)
	\end{align}
	since $\chi$ is compactly supported. We also know from Lemma \ref{ES1} (with $\varrho$ instead of $R$) that $Z_{B_k}$ is twice continuously differentiable with respect to $z$ and
	\begin{align*}
	\big\|  t\mapsto Z_{B_k}(s,t,\cdot)  \big \|_{L^\infty(0,T;H^2(B_\varrho(0)))} \le C,\quad s\in[0,T].
	\end{align*}
	Now recall the implicit representation formula \eqref{EXPLVL2} for $h_{B_k}$ that is\vspace{-1mm}
	\begin{align}
	\label{EQHBK}
	h_{B_k}(t,z) = f_d\big(Z_{B_k}(T,t,z)\big) - \int\limits_t^T \big[ \PHI_{f_{B_k},h_{B_k}}\, \chi \big]\big( s,Z_{B_k}(s,t,z) \big) \ds
	\end{align}
	for all $(t,z) \in [0,T]\times \RR^6$. As $f_d\in C^2_c(\RR^6)$ and $Z_{B_k}(T,t,\cdot)\in C^2(\RR^6)$, the term $f_d(Z_{B_k}(T,t,z))$ is twice continuously differentiable with respect to $z$ by chain rule. 
	By approximating $\PHI_{f_{B_k},h_{B_k}}\,\chi$ by sufficiently smooth functions one can easily show that the integral term of \eqref{EQHBK} is twice weakly differentiable and the derivatives can be computed by chain rule (with weak instead of classical derivatives if necessary). Now, one can show that the weak derivative $\delzi\delzj h_{B_k}$ can be bounded by
	\begin{align*}
	\|\delzi\delzj h_{B_k}(t)\|_{L^2}
	%
	&\le C + C\,\|Z_{B_k}(0) \|_{H^2(B_{\varrho}(0))} + C\,\int\limits_0^T  \|Z_{B_k}(s)\|_{H^2(B_{\varrho}(0))} \ds\;.
	\end{align*}
	By \eqref{EQZBK} this finally yields $\|\delzi\delzj h_{B_k}\|_{L^\infty(0,T;L^2)}^2 \le C$. Then $(\delzi\delzj h_{B_k})$ is converging with respect to the weak-*-topology on $[L^1(0,T;L^2)]^*=L^\infty(0,T;L^2)$ up to a subsequence. Because of uniqueness, the weak-*-limit of the sequence $(\delzi\delzj h_{B_k})$ must be $\delzi\delzj h_{B}$ and especially $h_B \in L^\infty(0,T;H^2)$. This completes the proof.
\end{proof}

Now inserting the state $f_B$ and its costate $g_B$ in \eqref{DJLAG} yields
\begin{align}
\label{EQFDJ}
J'(B)[H] = (\partial_B \Lag)(f_B,B,g_B)\big[H\big],\quad H\in\VV
\end{align}
since $f_B'[H]\big\vert_{t=0} = 0$. This provides a necessary optimality condition:

\begin{theorem}
	\label{NOC1}
	\hypertarget{HNOC1} $\;$
	\begin{itemize}
		\item[\textnormal{(a)}]
		The Fréchet derivative of $J$ at the point $B\in\BB$ is given by 
		\begin{align*}
		J'(B)[H] = 	\hspace{-5pt}\int\limits_{[0,T]\times\RR^3}\hspace{-5pt} \left( -\lambda \laplace_x B + \int\limits_{\RR^3} v\times \delv f_B \; g_B \dv \right)\cdot H \dtx, \quad H\in\VV.
		\end{align*}
		%
		\item[\textnormal{(b)}]
		Let us assume that $\B\in\BB$ is a locally optimal solution of the optimization problem \eqref{OP1}. Then for all $B\in\BB$,\vspace{-1mm}
		\begin{align*}
		\int\limits_{[0,T]\times\RR^3} \hspace{-7pt} \left( -\lambda \laplace_x \B + \int\limits_{\RR^3} v\stimes \delv f_\B \; g_\B \dv \right)\hspace{-2pt}\cdot\hspace{-1pt} (B-\bar B) \; \mathrm d(t,x)
		\begin{cases}
		= 0, &\hspace{-5pt} \text{if } \B \in \IBB \\
		\ge 0, &\hspace{-5pt} \text{if } \B \in \partial\BB
		\end{cases}\hspace{-5pt}.
		\end{align*}
		\item[\textnormal{(c)}]
		If we additionally assume that $\B\in\IBB$ then $\B$ satisfies the semilinear Poisson equation
		\begin{align}
		\label{FRJ}
		-\laplace_x \B = -\frac 1 \lambda \int\limits_{\RR^3} v\times \delv f_\B \; g_\B \dv
		.	\end{align}
		In this case $\B\in C([0,T];C^2_b(\RR^3))$ with
		\begin{align}
		\label{FRJ-2}
		\B(t,x) = -\frac 1 {4\pi\lambda} \iint \frac{1}{|x-y|}\; w\times \delv f_\B(t,y,w)\; g_\B(t,y,w) \;\mathrm dw \mathrm dy
		\end{align}
		for all $t\in[0,T]$ and $x\in\RR^3$. Thus $\B$ does not depend on the choice of $\chi$ as long as $\chi = 1$ on $\BRZ$ as it only depends on $g_\B\big\vert_\BR$. 
	\end{itemize}
\end{theorem}

\begin{proof} (a) follows immediately from \eqref{DBLAG} and \eqref{EQFDJ}. (b) is a direct consequence of Lemma \ref{FDJ-1} and (a) with ${H:=B-\bar B}$ and (b) implies \eqref{FRJ}. Recall that for almost all $t\in [0,T]$, $\B(t)$ has a continuous representative satisfying $\B_i(t,x) \to 0$ if $|x|\to\infty$ for every $i\in\{1,2,3\}$. Hence $\B$ is uniquely determined by \eqref{FRJ-2}.
	We must still prove that $\B$ lies in $C([0,T];C^2_b(\RR^3))$. Recall that $f_\B$ and $g_\B$ are in $ W^{1,2}(0,T;C_b(\RR^6))$ $\cap\, C([0,T];C_b^1(\RR^6))$ as $\B\in\BB$. Thus
	\begin{align*}
	p:[0,T]\times\RR^3 \to\RR^3,\quad (t,x)\mapsto \int\limits v\times \delv f_\B \; g_\B \dv
	\end{align*}
	is continuous with $\supp p(t)\subset \BR$ for all $t\in[0,T]$. By approximating $f_\B$ by $C([0,T];C^2_b)$-functions and using integration by parts one can easily show that $p$ is continuously differentiable whereby the partial derivatives are given by
	\begin{align*}
	\delxi p =  - \int\limits (v\times \delv f_\B) \; \delxi g_\B - (v\times \delv g_\B) \; \delxi f_\B\dv, \quad i=1,2,3.
	\end{align*}
	Consequently $\B\in C([0,T];C^2_b(\RR^3;\RR^3))$. Since $g_\B$ does not depend on $\chi$ as long as $\chi=1$ on $\BRZ$ the same holds for $\B$. \end{proof}

Note that Theorem \ref{NOC1} provides only a necessary but not a sufficient condition for local optimality. If a control $B$ satisfies the above condition it could still be a saddle point or even a local maximum point. Theorem \ref{NOC1} does also not provide uniqueness of the locally optimal solution. However the globally optimal solution that is predicted by Theorem \ref{EXOC1} is also locally optimal. Thus we have at least one control to satisfy the necessary optimality condition of Theorem \ref{NOC1}. \pskip

Assuming that there exists a locally optimal solution $\B\in\IBB$ we can easily deduce from Theorem \ref{NOC1} that the triple $(f_\B,g_\B,\B)$ is a classical solution of some certain system of equations.

\begin{corollary}
	\label{OSY1}
	Suppose that $\B\in\IBB$ is a locally optimal solution of the optimization problem \eqref{OP1}. Let $f_\B$ and $g_\B$ be its induced state and costate. Then $f_\B, g_\B\in C^1([0,T]\times\RR^6)$  and the triple
	$(f_\B,g_\B,\B)$ is a classical solution of the \textbf{optimality system}
	\begin{align}
	\label{OS-SEQ0}
	\begin{cases}
	\delt f + v\scdot \delx f - \delx\psi_f\scdot\delv f + (v\stimes B)\scdot\delv f = 0, &\hspace{-50pt} f\big\vert_{t=0} = \mathring f\\[0.15cm]
	\delt g + v\scdot \delx g - \delx\psi_f\scdot\delv g + (v\stimes B)\scdot\delv g = \PHI_{f,g}\chi, &\hspace{-50pt} g\big\vert_{t=T} = f(T)-f_d\\[0.15cm]
	B(t,x) = -\frac 1 {4\pi\lambda} \iint \frac{1}{|x-y|}\; w\times \delv f(t,y,w)\; g(t,y,w) \dyw\;.
	\end{cases}
	\end{align}
	For all $t\in [0,T]$, $\supp f_\B(t) \subset \BR$ and $\supp g_{\B}(t) \subset B_{R^*}(0)$.
\end{corollary}

\begin{proof}
	From Theorem \ref{NOC1} we know that $\B\in C([0,T];C^2_b)$. Thus by Theorem \ref{GCS} the solution $f_\B$ is classical, lies in $C^1([0,T]\times\RR^6)\cap C([0,T];C^2_b)$ and satisfies $\supp f_\B(t) \subset \BR$,\linebreak $t\in[0,T]$. We can use the decomposition $g_\B=f_\B+h_\B$ from the proof of Theorem \ref{ADJS} and from Proposition \ref{CSLVL} we can easily conclude that $g_\B$ is classical, i.e., $g_\B\in C^1([0,T]\times\RR^6)$ with $\supp g_\B(t) \subset B_{R^*}(0)$, $t\in[0,T]$. The rest is obvious due to the construction of $f_\B$, $g_\B$ and Theorem~\ref{NOC1}.
\end{proof}

\subsection{A sufficient condition for local optimality}


To prove that our cost functional is twice continuously Fréchet differentiable we will need Fréchet differentiability of first order of the costate.

\begin{lemma}
	\label{FDCCSO}
	\hypertarget{HFDCCSO}
	Let $g.:\BB\to C([0,T];L^2(\RR^6)),\; B \mapsto g_B$ denote the field-costate operator. For any $B\in\BB$ and $H\in\VV$ there exists a unique strong solution $g_B^H\in {H^1(]0,T[\times\RR^6)}$ of the final value problem
	\begin{displaymath}
	\refstepcounter{equation}
	\label{GDEQ}
	\left\{
	\begin{aligned}
	&\delt g + v\scdot\delx g - \delx\psi_{f'_B[H]}\scdot\delv g_B - \delx\psi_{f_B}\scdot\delv g + (v\stimes B)\scdot\delv g + (v\stimes H)\cdot\delv g_B \\
	&\qquad = \PHI_{f_B,g}\chi - \PHI_{g_B,f_B'[H]}\chi \\[0.2cm]
	&g\big\vert_{t=T}=0\;. 
	\end{aligned} 
	\right.\hspace{-15pt} \textnormal{(\theequation)}
	\end{displaymath}
	Then the following holds:
	\begin{itemize}
		\item [\textnormal{(a)}] The control-costate operator $g.$ is Fr\'echet differentiable on $\IBB$ with respect to the $C([0,T];L^2(\RR^6))$-norm, i.e., for any $B\in\IBB$ there exists a unique linear operator $g'_B: \VV\to C([0,T];L^2(\RR^6))$ such that\\
		$$\forall \eps>0\; \exists \delta > 0\; \forall H\in \VV \text{ with } \|H\|_{\VV} < \delta :$$
		$$B+H \in \IBB \tand \frac{\| g_{B+H} - g_B - g'_B[H] \|_{C([0,T];L^2)}}{\|H\|_{\VV}} < \eps.$$
		The Fr\'echet derivative is given by $g'_B[H]=g_B^H$ for all $H\in\VV$.
		\item [\textnormal{(b)}] For all $B,H\in\IBB$, the solution $g_B^H$ depends Hölder-continuously on $B$ in such a way that there exists some constant $C>0$ depending only on $\mathring f, T, K$ and $\beta$ such that for all $A,B\in\IBB$,
		\begin{align}
		\label{CFG}
		\underset{\|H\|_\VV\le 1}{\sup}\;\|g_A'[H] - g_B'[H]\|_{L^2(0,T;L^2)} \le C \;\|A-B\|_{\VV}^{\gamma}.
		\end{align}
	\end{itemize}
\end{lemma}

The proof proceeds analogously to the proof of Theorem \ref{FDCSO}.

\begin{remark}
	As $K$ was arbitrary the above results hold true if $\IBB$ is replaced by~$\IBBB$. Hence they are especially true on $\BB$.
\end{remark}

Continuous differentiability of the cost functional then follows:

\begin{corollary}
	\label{TFDJ}
	\hypertarget{HTFDJ}
	The cost functional $J$ of the optimization problem \eqref{OP1} is twice Fréchet differentiable on $\IBB$. The Fréchet derivative of second order at the point $B\in\IBB$ can be described as a bilinear operator $J''(B):\VV^2\to \RR$ that is given by
	\begin{align*}
	J''(B)[H_1,H_2] &= \lambda\; \langle D_x H_1, D_x H_2 \rangle_{L^2([0,T]\times\RR^3)}\\
	&\quad - \int\limits_{[0,T]\times\RR^6} (v\times H_1)\cdot \big( \delv f_B\; g_B'[H_2] - \delv g_B\; f_B'[H_2] \big) \; \mathrm d(t,x,v)
	\end{align*}
	for all $H_1,H_2\in\IBB$. Moreover there exists some constant $C>0$ depending only on $\mathring f$, $f_d$, $T$, $K$ and $\beta$ such that for all $B,\tB\in\IBB$,
	\begin{align*}
	\| J''(B) - J''(\tB) \| \le C\, \|B-\tB\|_{\VV}^\gamma
	\end{align*}
	where
	\begin{align*}
	\| J''(B) \| = \sup\Big\{ \big|J''(B)[H_1,H_2] \big| \,\Big\vert\, \|H_1\|_{\VV} = 1 ,\, \|H_2\|_{\VV} = 1\Big\}
	\end{align*}
	denotes the operator norm. This means that $J$ is twice continuously differentiable.
\end{corollary}

\noindent This result can be proved using the decompositions $f_{B+H}-f_B = f'_B[H] + f_R[H]$ and $g_{B+H}-g_B = g'_B[H] + g_R[H]$ from Theorem \ref{FDCSO} and Theorem~\ref{FDCCSO} where
\begin{align*}
	\|f_R[H]\|_{C([0,T];L^2)} = \text{o}(\|H\|_{\VV}),\qquad  \|g_R[H]\|_{C([0,T];L^2)} = \text{o}(\|H\|_{\VV}).
\end{align*}

\pagebreak[1]
The following theorem provides a sufficient condition for local optimality:
\begin{theorem}
	\label{QGC1}
	\hypertarget{HQGC1}
	Suppose that $\B \in\BB$ and let $f_\B$ and $g_\B$ be its induced state and costate. Let $0<\alpha<2+\gamma$ be any real number. We assume that the variation inequality
	\begin{align}
	\label{VIQ}
	\int\limits_{[0,T]\times\RR^3} \hspace{-0.2cm}\left( -\lambda \laplace_x \B + \int\limits_{\RR^3} v\times \delv f_\B \; g_\B \dv \right)\cdot (B-\B) \; \mathrm d(t,x) = J'(\B)[B-\B] \ge 0
	\end{align}
	holds for all ${B\in\BB}$ and that there exists some constant $\eps>0$ such that 
	\begin{displaymath}
	\refstepcounter{equation}
	\begin{aligned}
	\label{CIQ}
	&\lambda\; \|D_x H\|_{L^2([0,T]\times\RR^3)}^2
	- \hspace{-0.2cm}\int\limits_{[0,T]\times\RR^6} (v\times H)\cdot \big( \delv f_\B\; g_\B'[H] - \delv g_\B\; f_\B'[H] \big) \; \mathrm d(t,x,v)\\[0.3cm]
	&\quad = \; J''(\B)[H,H] \;  \ge \eps \; \|H\|_{\VV}^\alpha \hspace{189pt} \textnormal{(\theequation)}
	\end{aligned}
	\end{displaymath}
	holds for all $H\in\BB$. Then $J$ satisfies the following growth condition: There exists $\delta>0$ such that for all $B\in\BB$ with $\|B-\B\|_{\VV}<\delta$,
	\begin{align}
	\label{QGC}
	J(B) \ge J(\B) + \frac \eps 4 \|B-\B\|_{\VV}^\alpha
	\end{align}
	and hence $\B$ is a strict local minimizer of $J$ on the set $\BB$.
\end{theorem}

\begin{proof} Let $B\in\BB$ be arbitrary. We define the auxillary function \linebreak $F:[0,1]\to\RR_0^+$, ${s \mapsto J\big(\B+s(B-\B)\big)}$. Then $F$ is twice continuously differentiable by chain rule and Taylor expansion yields $F(1) = F(0) + F'(0) + \tfrac 1 2 F''(\vartheta)$ for some $\vartheta\in ]0,1[$. By the definition of $F$ this implies that
	\begin{align*}
	J\big(B\big) 
	&\ge  J\big(\B\big) +  \tfrac 1 2 J''\big(\B\big)[B-\B,B-\B] \\
	&\quad +  \tfrac 1 2 \Big( J''\big(\B+\vartheta (B-\B)\big) - J''\big(\B\big) \Big)[B-\B,B-\B]
	\end{align*}
	Now, according to Corollary \ref{TFDJ},
	\begin{align*}
	&\Big| \Big( J''\big(\B+\vartheta (B-\B)\big) - J''\big(\B\big) \Big)[B-\B,B-\B] \Big| \le C\, \|B-\B\|_{\VV}^{2+\gamma}
	\end{align*}
	Suppose now that $ \|B-\B\|_{\VV}<\delta$ for some $\delta>0$. Then
	\begin{align*}
	&\Big| \Big( J''\big(\B+\vartheta (B-\B)\big) - J''\big(\B\big) \Big)[B-\B,B-\B] \Big| \\
	&\quad \le \; C\, \delta^{2+\gamma-\alpha} \|B-\B\|_{\VV}^\alpha \;\le\; \frac{\eps}{2} \|B-\B\|_{\VV}^\alpha
	\end{align*}
	if $\delta$ is sufficiently small. In this case $J\big(B\big) \ge J\big(\B\big) + (\eps/4) \|B-\B\|_{\VV}^\alpha$. This especially means that $J(B)>J(\B)$ for all $B\in B_\delta (\B) \cap \BB$ and consequently $\B$ is a strict local minimizer of $J$. \end{proof}

\subsection{Uniqueness of the optimal solution on small time intervals}

We know from Corollary \ref{OSY1} that for any locally optimal solution $\B\in\IBB$ the triple $(f_\B,g_\B,\B)$ is a classical solution of the optimality system
\begin{align}
\label{OS1}
\left\{
\begin{aligned}
&\delt f +v\scdot \delx f - \delx\psi_f\scdot\delv f + (v\stimes B)\scdot\delv f = 0, &&\hspace{-50pt} f\big\vert_{t=0} = \mathring f\\[0.15cm]
&\delt g +v\scdot \delx g - \delx\psi_f\scdot\delv g + (v\stimes B)\scdot\delv g = \PHI_{f,g}, &&\hspace{-50pt} g\big\vert_{t=T} = f(T)-f_d\\[0.15cm]
&B(t,x) = -\frac 1 {4\pi\lambda} \textstyle\iint \frac{1}{|x-y|}\; w\times \delv f(t,y,w)\; g(t,y,w) \dyw \;.\\[0.15cm]
\end{aligned}
\right.
\end{align}
The following theorem states that the solution of this system of equations is unique if the final time $T$ is small compared to $\lambda$. As we will have to adjust $\frac T \lambda$ it is necessary to assume that $0<\lambda\le \lambda_0$ for some constant $\lambda_0>0$. Of course large regularaization parameters $\lambda$ do not make sense in our model, so we will just assume that $\lambda_0 = 1$.

\begin{theorem}
	\label{UOS}
	\hypertarget{HUOS}
	Suppose that $\lambda\in ]0,1]$ and suppose that there exists a classical solution $(f,g,B)$ of the optimality system \eqref{OS1}, i.e., $B\in C\big([0,T];C^1_b(\RR^3;\RR^3)\big)$, $f,g\in C^1([0,T]\times\RR^6)$ with $\supp f(t),\,\supp g(t) \subset B_r(0)$ for some radius $r>0$. Then this solution is unique if the quotient $\tfrac T \lambda$ is sufficiently small.
\end{theorem}

\begin{proof}
	Suppose that the triple $(\tf,\tg,\tB)$ is another classical solution that is satisfying the support condition with radius $\tilde r$. Without loss of generality we assume that $r=\tilde r$. Let $C=C(T)\ge 0$ denote some generic constant that may depend on $T$, $\mathring f$, $f_d$, $r$, $\|\chi\|_{C^1_b}$ and the $C([0,T];C^1_b)$-norm of $f$, $\tf$, $g$ and $\tg$. We can assume that $C=C(T)$ is monotonically increasing in $T$. First of all, by integration by parts,
	\begin{align}
	\label{UNIEU}
	&\|B(t)-\tB(t)\|_\infty \le  \frac C \lambda \|g(t)-\tg(t)\|_\infty +  \frac C \lambda \|f(t)-\tf(t)\|_\infty, \quad t\in [0,T].
	\end{align}
	Let now $Z$ and $\tZ$ denote the solutions of the characteristic system of the Vlasov equation to the fields $B$ and $\tB$ satisfying $Z(t,t,z) = z$ and $\tZ(t,t,z)=z$ for any $t\in[0,T]$ and $z\in\RR^6$. Then for any $s,t\in [0,T]$ (where $s\le t$ without loss of generality) and $z\in\RR^6$,\newpage
	\begin{align*}
	&|Z(s,t,z)-\tZ(s,t,z)| \\
	&\quad \le \int\limits_s^t C\; |Z(\tau,t,z)-\tZ(\tau,t,z)| + \tfrac C \lambda\; \|f(\tau)-\tf(\tau)\|_\infty + \tfrac C \lambda\; \|g(\tau)-\tg(\tau)\|_\infty \dtau
	\end{align*}
	and hence
	\begin{align}
	\label{EQ0}
	|Z(s,t,z)-\tZ(s,t,z)| \le C \int\limits_s^t \hspace{-3pt}\tfrac 1 \lambda\, \|f(\tau)-\tf(\tau)\|_\infty + \tfrac 1 \lambda\, \|g(\tau)-\tg(\tau)\|_\infty \dtau
	\end{align}
	by Gronwall's lemma. Consequently
	\begin{align*}
	\|f(t)-\tf(t)\|_\infty &\le C\; \|Z(0,t,\cdot)-\tZ(0,t,\cdot)\|_\infty \\
	&\le C\; \int\limits_0^t \tfrac 1 \lambda\; \|f(\tau)-\tf(\tau)\|_\infty + \tfrac 1 \lambda\; \|g(\tau)-\tg(\tau)\|_\infty \dtau
	\end{align*}
	which yields\vspace{-2mm}
	\begin{align*}
	\|f(t)-\tf(t)\|_\infty \le C\; \tfrac 1 \lambda \exp\left( C\;\tfrac T \lambda \right)\; \int\limits_0^t \|g(\tau)-\tg(\tau)\|_\infty \dtau
	\end{align*}
	and thus
	\begin{align}
	\label{UNIEF}
	\|f-\tf\|_{C([0,T];C_b)} \le C\; \tfrac T \lambda \exp\left( C\;\tfrac T \lambda \right)\; \|g-\tg\|_{C([0,T];C_b)}\,.
	\end{align}
	For $z\in B_r(0)$ and $t\in[0,T]$ we can conclude from \eqref{EXPLVL} that
	\begin{align*}
	&|g(t,z)-\tg(t,z)| \\
	&\quad\le C\; \|Z(T,t,\cdot)-\tZ(T,t,\cdot)\|_\infty + \int\limits_{t}^T \|\PHI_{f,g}(\tau) - \PHI_{\tf,\tg}(\tau)\|_{L^\infty(B_r(0))} \dtau\\
	&\qquad + C\;\int\limits_{t}^T \|\PHI_{\tf,\tg}(\tau)\|_{W^{1,\infty}} \|Z(\tau,t,\cdot) - \tilde Z(\tau,t,\cdot)\|_\infty \dtau\;.
	\end{align*}
	We already know from inequality \eqref{EQ0} that for $t\le\tau\le T$,
	\begin{align*}
	\|Z(\tau,t,\cdot) - \tilde Z(\tau,t,\cdot)\|_\infty  &\le C\; \int\limits_t^\tau \tfrac 1 \lambda\; \|f(\sigma)-\tf(\sigma)\|_\infty + \tfrac 1 \lambda\; \|g(\sigma)-\tg(\sigma)\|_\infty \;\mathrm d \sigma\,.
	\end{align*}
	\newpage
	Also recall that
	\begin{align*}
	&\|\PHI_{f,g}(\tau)\|_{W^{1,\infty}} \le \|\PHI_{f,g}(\tau)\|_\infty + \|\PHI'_{f,g}(\tau)\|_\infty \\
	&\quad \le C\; \|f\|_{C([0,T];C^1_b)}\, \|g\|_{C([0,T];C^1_b)} \le C
	\end{align*}
	for every $\tau\in[0,T]$. Moreover, by \eqref{ESTPHI3},
	\begin{align*}
	&\|\PHI_{f,g}(\tau)-\PHI_{\tf,\tg}(\tau)\|_{L^\infty(B_r(0))} \le C\; \|f(\tau)-\tf(\tau)\|_\infty + C\; \|g(\tau)-\tg(\tau)\|_\infty
	\end{align*}
	for all $\tau\in[0,T]$. This implies that for all $t\in[0,T]$,
	\begin{align*}
	\|g(t)-\tg(t)\|_\infty
	&\le C\; \int\limits_t^T \tfrac 1 \lambda \; \|g(\tau)-\tg(\tau)\|_\infty + \tfrac 1 \lambda  \; \|f(\tau)-\tf(\tau)\|_\infty \dtau
	\end{align*}
	and hence
	\begin{align}
	\label{UNIEG}
	\|g-\tg\|_{C([0,T];C_b)} \le C\; \tfrac T \lambda \exp\left( C\;\tfrac T \lambda \right)\; \|f-\tf\|_{C([0,T];C_b)}
	\end{align}
	by Gronwall's lemma. Inserting \eqref{UNIEG} in \eqref{UNIEF} yields
	\begin{align*}
	\|f-\tf\|_{C([0,T];C_b)} \le C \left(\tfrac T \lambda\right)^2 \exp\left( C\;\tfrac T \lambda \right)\; \|f-\tf\|_{C([0,T];C_b)}\;.
	\end{align*}
	If now $\tfrac T \lambda$ is sufficiently small we have $C \left(\tfrac T \lambda\right)^2 \exp\left( C\;\tfrac T \lambda \right)<1$ and we can conclude that $f=\tf$ on $[0,T]\times\RR^6$. Then obviously $g=\tg$ by \eqref{UNIEG} and $B=\tB$ by \eqref{UNIEU} which means uniqueness of the solution $(f,g,B)$. \end{proof}

If $\B\in\IBB$ is locally optimal, the following uniqueness result holds:

\begin{corollary}
	Suppose that $\lambda\in]0,1]$ and let $\B\in\IBB$ be a locally optimal solution of the optimization problem \eqref{OP1}. Then the tripel $(f_{\B},g_{\B},\B)$ is a classical solution of the optimality system \eqref{OS1} according to Corollary \ref{OSY1}.\pskip
	
	If now $\lambda\in]0,1]$ and $\tfrac T \lambda$ is sufficiently small then $\B$ is the only locally optimal solution of the optimization problem \eqref{OP1} in $\IBB$. \pskip
	
	Suppose that there is a globally optimal solution $B\in\IBB$. Then $B=\B$ is the unique globally optimal solution in $\IBB$. However it is still possible that there are other globally optimal solutions in $\partial\BB$.
\end{corollary}

\begin{proof} If $\lambda\in]0,1]$ and $\tfrac T \lambda$ is sufficiently small then Proposition \ref{UOS} ensures that $\B$ is the only locally optimal solution. Recall that there exists at least one globally optimal solution according to Theorem \ref{EXOC1}. Let us assume that $B\in\IBB$ in one of these globally optimal solutions. As any globally optimal solution is also locally optimal it follows that there is only one globally optimal solution in $\IBB$ and thus $B=\B$. \end{proof}


\footnotesize
\begin{thebibliography}{00}
\bibitem{batt} J. Batt, \textit{Global symmetric solutions of the initial value problem in stellar dynamics}, J. Differential Equations \textbf{25} (1977), 342-364.
\bibitem{dragomir} S.S. Dragomir, \textit{Some Gronwall Type Inequalities and Applications}, Nova Science Publishers (2003).
\bibitem{evans} L.C. Evans, \textit{Partial Differential Equations (Second Edition)}, Amer. Math. Soc., Grad. Stud. in Math. \textbf{19} (2010).
\bibitem{gilbarg-trudinger} D. Gilbarg, N.S. Trudinger, \textit{Elliptic Partial Differential Equations of Second Order}, Springer \textbf{3} (2001)
\bibitem{kurth} R. Kurth,\,\textit{Das Anfangswertproblem der Stellardynamik},\,Z.\,Astrophys.\,\textbf{30}\,(1952),\,213-229.
\bibitem{kreuter} M. Kreuter, \textit{Sobolev spaces of vector-valued functions}, Master Thesis, Ulm University.
\bibitem{lieb-loss} E.H. Lieb and M. Loss, \textit{Analysis (Second Edition)}, Amer. Math. Soc. \textbf{14} (2001).
\bibitem{lions-perthame} P.-L. Lions and B. Perthame, \textit{Propagation of moments and regularity for the \linebreak3-dimensional Vlasov-Poisson system}, Invent. Math. \textbf{105} (1991), 415-430.
\bibitem{pettis} B.J. Pettis, \textit{On integration in vector spaces}, Trans. Amer. Math. Soc. \textbf{44}, 277-304 (1938).
\bibitem{pfaffelmoser} K. Pfaffelmoser, \textit{Global classical solutions of the Vlasov-Poisson system in three dimensions for general initial data}, J. Differential Equations \textbf{95} (1992), 281-303.
\bibitem{rein} G. Rein, \textit{Collisionless Kinetic Equations from Astrophysics - The Vlasov-Poisson System}, Handbook of Diff. Equations: Evolutionary Equations, Elsevier B.V. \textbf{3} (2007).
\bibitem{schaeffer} J. Schaeffer, \textit{Global existence of smooth solutions to the Vlasov-Poisson system in three dimensions}, Comm. Partial Differential Equations \textbf{16} (1991), 1313-1335.
\bibitem{stein} E. Stein, \textit{Singular integrals and differentiability properties of functions}, Princeton Univ. Press (1970)
\bibitem{yosida} K. Yosida, \textit{Functional Analysis}, Springer \textbf{6} (1980).
\end{thebibliography}
\end{document}